\documentclass[11pt]{amsart}
\usepackage{latexsym,amssymb,amsmath}
\usepackage[colorlinks=true, linkcolor=blue, anchorcolor=black, citecolor=blue, filecolor=blue, menucolor= blue, urlcolor=black]{hyperref}
\usepackage{graphicx}
\usepackage{float}
\usepackage{faktor}
\usepackage[all]{xy}
\usepackage{tikz}
\usepackage{enumitem}

\newtheorem{theorem}{Theorem}[section]
\newtheorem{lemma}[theorem]{Lemma}
\newtheorem{proposition}[theorem]{Proposition}
\newtheorem{corollary}[theorem]{Corollary}

\theoremstyle{definition}
\newtheorem{definition}[theorem]{Definition}
\newtheorem{terminology}[theorem]{Terminology}
\newtheorem{procedure}[theorem]{\emph{Macaulay2} Code}
\newtheorem{remark}[theorem]{Remark}
\newtheorem{example}[theorem]{Example}
\newtheorem{question}[theorem]{Question}

\newtheorem{theoremx}{Theorem}

\newtheorem{corollarymx}[theoremx]{Corollary}

\makeatletter
\@namedef{subjclassname@2020}{%
  \textup{2020} Mathematics Subject Classification}
\makeatother

\newcommand{\Ass}{\operatorname{Ass}}

\newcommand{\IM}{\operatorname{im}}

\newcommand{\car}{\operatorname{char}}

\DeclareMathOperator{\init}{in_{\prec}}
\newcommand{\reg}{\operatorname{reg}}

\begin{document}


\title{Connected domination in graphs and v-numbers of binomial edge ideals}

\author[D. Jaramillo-Velez]{Delio Jaramillo-Velez}
\address{
Departamento de
Matem\'aticas\\
Centro de Investigaci\'on y de Estudios
Avanzados del
IPN\\
Apartado Postal
14--740 \\
07000 Mexico City, CDMX.
}
\email{djaramillo@math.cinvestav.mx}
\author[L. Seccia]{Lisa Seccia}
\address{Max Planck Institute for Mathematics in the Sciences, Leipzig, Germany.}
\email{seccia@mis.mpg.de}


\thanks{The first author is supported by CONACyT Fellowship 862006. The second author was partially supported by PRIN 2020355BBY ``Squarefree Gr\"obner degenerations, special varieties and related topics"}


\maketitle

\begin{abstract}
The v-number of a graded ideal is an algebraic invariant introduced by Cooper et al., and originally motivated by problems in algebraic coding theory. In this paper we study the case of binomial edge ideals and we establish a significant connection between their v-numbers and the concept of connected domination in graphs. 
More specifically, we prove that the localization of the v-number at one of the minimal primes of the binomial edge ideal $J_G$ of a graph $G$ coincides with the connected domination number of the defining graph, providing a first algebraic description of the connected domination number. As an immediate corollary, we obtain a sharp combinatorial upper bound for the v-number of binomial edge ideals of graphs. Lastly, building on some known results on edge ideals, we analyse how the v-number of $J_G$ behaves under Gr\"obner degeneration when $G$ is a closed graph. 

\end{abstract}


\section{Introduction}\label{intro-section}

Let $K$ be a field, and  let $S=K[t_{1},\dots,t_{n}]=\bigoplus_{d=0}^{\infty}S_{d}$ be the polynomial ring in $n$ variables with the standard grading. Given a homogeneous ideal $I$ of $S$, the v-{\em number} of $I$ is the following invariant
$$
{\rm v}(I):=\min\{d\geq 0 \mid\, \exists\, f 
\in S_d \mbox{ and }\mathfrak{p} \in {\rm Ass}(I) \mbox{ with } (I\colon f)
=\mathfrak{p}\}.
$$
If we define  the v-number  of $I$  locally  at each associated prime  $\mathfrak{p}$ of $I$ as 
$$
{\rm v}_{\mathfrak{p}}(I):=\min\{d\geq 0 \mid\, \exists\, f 
\in S_d  \mbox{ with } (I\colon f)
=\mathfrak{p}\},
$$
it follows from the definition that $
{\rm v}(I)=\min\{{\rm v}_{\mathfrak{p}}(I) \mid \mathfrak{p} \in \Ass(I)\}.
$ We call ${\rm v}_{\mathfrak{p}}(I)$ the \emph{localization} of ${\rm v}(I)$ at $\mathfrak{p}$.
From a geometric point of view, the local definition of v-numbers extends the notion of degree of a point in a finite set of projective points given in \cite{GKR}. 
However, one of the main reasons that led to a systematic study of v-numbers comes from coding theory.
In this context, Cooper et al. \cite{CSTPV} introduced ${\rm v}$-numbers of graded ideals in order to investigate the asymptotic behavior of the \emph{minimum distance function} of projective Reed-Muller-type codes. \par 
Regardless of its meaning in coding theory, the minimum distance function $\delta _I$ (over the  integers) is defined for any graded ideal in terms of the Hilbert-Samuel multiplicity of $S/I$ (see \cite{NBPV18}, \cite{CSTPV}). 
If $I$ is radical, this function is non-increasing \cite{NBPV20} and one can define the \emph{regularity index of} $\delta _I$, denoted by $\text{reg}(\delta _I)$, as the smallest integer for which $\delta _I$ stabilizes. 
From the viewpoint of algebraic coding theory, determining the regularity index is important since non-trivial evaluation codes can occur only if the degrees of the defining polynomials are smaller than the index of regularity.
It turns out that,
given a finite set of projective points $\mathbb{X}$, the v-number of its vanishing ideal $I(\mathbb{X})$ coincides with the index of regularity of the Reed-Muller type codes $\{C_{\mathbb{X}}(d)\}_{d=1}^{\infty}$ associated to $\mathbb{X}$, that is $\text{reg}(\delta _{I(\mathbb{X})})={\rm v} (I(\mathbb{X}))$  {\rm \cite{CSTPV}}, where $\delta_{I(\mathbb{X})}(d)$ is the minimum distance of $C_{\mathbb{X}}(d)$ in the classical sense of coding theory. \par
So far, in addition to vanishing ideals, v-numbers of squarefree monomial ideals have been broadly studied. These ideals have an intrinsic combinatorial nature and their algebraic properties often carry information on their underlying combinatorial structures, and viceversa. In \cite{DR} the authors consider these ideals as edge ideals of clutters and find a combinatorial expression for their ${\rm v}$-numbers. This combinatorial description has been exploited to classify $W_2$ graphs (see \cite[Theorem 4.2]{DR}) and to study the combinatorial structure of graphs whose edge ideals have Cohen-Macaulay second symbolic power. Thanks to the one-to-one correspondence between clutters and simplicial complexes, the combinatorial expression of v-numbers of edge ideals given in \cite{DR} can be also rephrased in terms of free faces of a simplicial complex \cite{Civan}.\par 
 On the algebraic side, the first author and Villarreal verify that under suitable assumptions ${\rm v}(I) \leq \reg (S/I)$, where $\reg (S/I)$ is the Castelnuovo-Mumford regularity.  Also, in \cite{Saha-Seng} the authors provide some classes of graphs such that ${\rm v}(I)\leq \IM (G) \leq \reg (S/I)$, where $\IM (G)$ is the induced matching number of $G$ (see Terminology \ref{T}) and $I$ is the edge ideal of  $G$. Thus, in many cases, the ${\rm v}$-number is a lower bound for $\reg (S/I)$. However, when $\car K=0$, there are examples where $\text{reg}(\delta _I)\geq {\rm v}(I)> \reg (S/I)$ \cite{DR}, and the difference ${\rm v}(I)-\reg (S/I)$ could be arbitrarily large \cite{Civan}. These examples also provide a counterexample to a  previous conjecture by Nu\~{n}ez-Betancourt, Pitones and Villarreal \cite[Conjecture 4.2]{NBPV18}. \\

Driven by these combinatorial and algebraic results, in this work we study the v-number of a binomial edge ideal $J_G$ associated to a graph $G$. 
Finding a general combinatorial description of the v-number of binomial edge ideals  turns out to be a more difficult task compared to the monomial case. Nevertheless, it is possible to simplify the problem by proving that the v-number of $J_G$ is additive on the connected components of $G$ (see Proposition \ref{additivity}). This reduction enables us to focus on the case of connected graphs where it is known that $J_{\mathcal{K}_{n}}$ is a minimal prime of $J_{G}$. Hence, by definition, the v-number of $J_G$ is bounded from above by its v-number at $J_{\mathcal{K}_{n}}$.  In Section \ref{v-num-bino-section}, we find a combinatorial description for ${\rm v}_{J_{\mathcal{K}_{n}}}(J_G)$, and so a combinatorial upper bound for ${\rm v}(J_G)$. Specifically, we show that ${\rm v}_{J_{\mathcal{K}_{n}}}(J_G)$ measures the connected domination of the graph. In doing so, we also extend {\rm \cite[Theorem 3.7]{Fatemeh-Leila}}  by Mohammadi and Sharifan (see Proposition \ref{comb-des-colo}).

A \emph{connected dominating set} (CDS) of a connected graph $G$ is a subset $D$ of the vertices of $G$ such that $D$ induces a connected graph $G_{D}$, and every vertex in $V(G) \setminus D$ is adjacent to at least one node in $D$. The minimum of the cardinalities of the connected dominating sets is the \emph{connected domination number}, denoted by $\gamma_c (G)$.

\begin{theoremx}[Theorem \ref{comb_at_comple}]\label{B}
Let $G\neq \mathcal{K}_n$ be a connected graph and let $J_{G}$ be its binomial edge ideal. Then 
 $$
{\rm v}_{J_{\mathcal{K}_{n}}}(J_{G})=\min\lbrace |B|\: |\: B\in\mathcal{D}_{c}(G)\rbrace=\gamma_c (G).$$
where $\mathcal{D}_{c}(G)= \left\lbrace B \subseteq [n] \mid \forall \lbrace i,j\rbrace\not\in E(G), \exists  \text{ a path } P:i,i_{1},\dots,i_{s},j \text{ s.t. }\lbrace i_{1},\dots,i_{s}\rbrace\subseteq B\right\rbrace$.
\end{theoremx}

Connected domination of graphs dates back to 1979 when it was formally introduced by  Sampathkumar and Walikar {\rm \cite{Sam-Wal}} as a variation of the concept of domination. Thereafter, it has been extensively studied in the literature also thanks to its numerous applications in network design, and wireless sensor networks. 
In this framework, CDSs play a crucial role in routing computation of Mobile Ad Hoc Networks (MANETs) where a minimum connected dominating set (MCDS) offers a way of sending messages that minimizes the energy consumption and prolongs the lifetime of the network. Likewise, in network design, the connected domination number can be used to determine the minimum number of nodes required to maintain a reliable connection between all the nodes in the network.
The problem of finding a MCDS is proven to be NP-complete, even when instances are restricted to bipartite graphs ({\rm \cite{Garey-Johnson}}, {\rm \cite{Pfaff}}).
For this reason, many algorithms have been developed to construct CDSs and MCDSs ({\rm \cite{guha}, \cite{li2005}, \cite{wan2004}}) and many upper and lower bounds were computed for the connected domination number of an arbitrary graph ({\rm \cite{Bo-Liu}, \cite{Wyatt}, \cite{karami}, \cite{Kosari}, \cite{mafuta}}). However, no closed formula is known to compute $\gamma_c (G)$. Theorem \ref{B} works as a bridge between network/graph theory and combinatorial algebra, as it gives an algebraic interpretation of the connected domination number locally in terms of a v-number. Based on this algebraic description, we provide a \emph{Macaulay2} function (see Code \ref{gammafunction}) to compute $\gamma_c (G)$. Hopefully, Theorem \ref{B} can lead to new improved bounds for the connected domination number.

It is worth noticing that connected domination is closely related to the problem of finding a spanning tree with the largest number of leaves. We recall that a \emph{spanning tree} of a graph $G$ is a subgraph that is a tree and covers all the vertices of $G$.  It is well-known that the connected domination number $\gamma_c (G)$ is the number of non-leaf nodes (i.e., vertices of degree at least $2$) in a spanning tree with the largest number of leaves \cite{caro-west-yuster}. In other words, $$ \gamma_c(G)=n-{\rm lf_{\max}}(G)$$ where ${\rm lf_{\max}}(G)$ is the \emph{maximum leaf number} of $G$. Therefore, as a corollary of the previous theorem, we obtain algebraic upper and lower bounds for $\gamma_c (G)$.

\begin{corollarymx}[Corollary \ref{corollaryMain}]\label{C}
If $G$ is a connected graph, then $${\rm v}(J_{G})\leq {\rm v}_{J_{\mathcal{K}_{n}}}(J_{G})=\gamma_c(G)\leq {\rm reg}(S/T_G),$$
where $T_G$ is a spanning tree of $G$ with the largest number of leaves.
\end{corollarymx}

In Subsection \ref{sectionclosed}, we use Theorem \ref{B} to investigate how v-numbers of binomial edge ideals of closed graphs behave with respect to Gr\"obner degeneration.
In this case, it is known that binomial edge ideals have a quadratic Gr\"obner basis and their initial ideal can be seen as the edge ideal of a bipartite graph. Thereby, using Gr\"obner degeneration and known results on v-numbers of monomial ideals, we can infer information on v-numbers of binomial edge ideals. Before stating the theorem, we recall below some preliminary notions.\par 
\begin{terminology} \label{T}
Let $G$ be a simple graph with set of vertices $V(G)=[n]:=\{1,2,\dots,n\}$, and set of edges $E(G)$.
\begin{itemize} [itemsep=1.3pt,topsep=2pt,leftmargin=0.3in] 
\item A set $\lbrace e_{1},\dots,e_{r}\rbrace$ of pairwise disjoint edges of $G$ is an \emph{induced matching} if the only edges of $G$ contained in $\bigcup_{i=1}^{r}e_{i}$ are $e_{1},\dots,e_{r}$. The {\it induced matching number} of $G$, denoted by $\IM (G)$, is the number of edges in the largest induced matching.
\item $\ell_G$ denotes the length of the longest induced path of $G$.
\item  A \emph{clique} of a graph is a nonempty subset $C \subseteq V(G)$ such that $C$ induces a complete graph. A \emph{clique cover} of $G$ is a family $\mathcal{C}=\{C_1, \ldots, C_s\}$ of cliques of $G$ such that for every $v\in V(G)$ there exists a clique $C\in \mathcal{C}$ with $v\in C$. The \emph{clique covering  number}  of $G$, denoted by $\theta(G)$, is the minimum number of cliques in $G$ needed to cover the vertex set of $G$, i.e.
$$
\theta(G):=\min\lbrace |\mathcal{C}|\:|\: \mathcal{C}\hbox{ is a clique cover of } G\rbrace.
$$
\end{itemize}
\end{terminology}

\begin{theoremx}[Theorem \ref{relation_initial}]\label{E}
Let $G$ be a connected closed graph. Then 

$$ 
{\rm v} (J_G)<\theta(G)\leq  {\rm v} (\init (J_G)) \leq \ell_G = \reg (S/J_G)=\reg (S/\init(J_G)).
$$
\end{theoremx}
We observe that the equality ${\rm reg}(S/J_{G})={\rm reg}(S/{\rm in}_{\prec}(J_{G}))=\ell_G$ follows from {\rm \cite{Via-Zar}}. \par
We also provide an example (see Example \ref{ExampleNC})  of a non-closed graph such that the v-number is not equal to the localization of the v-number at the complete graph, and ${\rm v}(\init (J_G)) > \ell_G$, but still $ {\rm v} (J_G)\leq {\rm v} (\init (J_G))\leq \reg (S/J_G)=\reg (S/\init(J_G)).$ Hence, in this example, ${\rm v}(\init (J_G))$ is a better bound for the regularity than $\ell$.

Lastly, as an application of previous results, in Section \ref{5} we compute the exact value of the v-number for the special family of complete multipartite graphs and we discuss some examples, using \emph{Macaulay2} Code \ref{v-number-procedure} \cite{Mac2}. All these examples suggest that the inequality 
${\rm v} (J_G)\leq  {\rm v} (\init (J_G))$ in Theorem \ref{E} holds for any graph $G$ (see Question \ref{question}).\\

 We would like to point out that some of the results presented in this paper can also be found in \cite{ambhore2023mathrm}. Nevertheless, the techniques and the approaches are significantly different, and the two works were carried out independently.

\section{Preliminaries}\label{preli-section}
In this section we collect some well-known properties and results on binomial edge ideals, edge ideals, and v-numbers, which are relevant for our purposes. The reader who is already familiar with these topics may wish to skip ahead to Section \ref{v-num-bino-section}.

\subsection{Binomial edge ideals}  
Binomial edge ideals were defined by Herzog, Hibi, Hreinsd\'ottir, Kahle and Rauh \cite{HHHKR} and, independently, by Ohtani \cite{Ohtani}, as a way to attach an algebraic structure to a graph. Since their inception, researchers have put a great effort in understanding the interplay between the combinatorial invariants of the graph and the algebraic invariants of its associated binomial edge ideal.
Nowadays, these ideals are considered a standard topic in combinatorial commutative algebra, and they also play a role in the rapidly-growing field of algebraic statistics, where they arise as ideals generated by conditional independence statements (see \cite{HHHKR}). \\

Let $S=K[x_1, \ldots,x_n, y_1, \ldots,y_n]$ be a polynomial ring over a field $K$, and let $G$ be a simple graph (i.e., $G$ is a graph which has neither multiple edges nor loops) on $n$ vertices with set of edges $E(G)$. The \emph{binomial edge ideal} of $G$, denoted by $J_G$, is the following quadratic binomial ideal
\begin{equation*}
J_G:= \left( f_{ij}:=x_i y_j-x_j y_i \mid \{i,j\} \in E(G), i<j\right).
\end{equation*}

In other words, $J_G$ is the ideal generated by the $2$-minors of the generic matrix
$$ X_{n}=
\begin{bmatrix}
    x_1       & x_2 & x_3 & \dots & x_n \\
     y_1       & y_2 & y_3 & \dots & y_n
\end{bmatrix}
$$

whose column indices are given by the edges of $G$. \par 
If we take $G=\mathcal{K}_n$ to be the complete graph on $n$ vertices, then it is clear from the definition that its binomial edge ideal $J_{\mathcal{K}_n}$ is the ideal of $2$-minors of $X_n$. This is why binomial edge ideals can be considered as a generalization of determinantal ideals.\par

In \cite{HHHKR} the authors prove that binomial edge ideals are radical and they give a combinatorial description of their associated primes.  Given a subset $\mathcal{S}\subseteq[n]$, one can define a prime ideal $P_{\mathcal{S}}$ as follows. Set $\mathcal{T}=[n]\setminus \mathcal{S}$, and  let $G_{1},\dots,G_{c(\mathcal{S})}$ be the connected components of $G_{\mathcal{T}}$, where $G_{\mathcal{T}}$ is the induced subgraph of $G$ on $ \mathcal{T}$. For each $G_{i}$, we denote  by $\widetilde{G}_{i}$ the complete graph on $V(G_{i})$ and we set
$$
P_{\mathcal{S}}(G):=\left( \bigcup_{i\in\mathcal{S}}\lbrace x_{i},y_{i}\rbrace, J_{\widetilde{G}_{1}},\dots,J_{\widetilde{G}_{c(\mathcal{S})}}\right). 
$$
This ideal is a prime ideal and all the associated primes of $J_G$ are of this form: 
\begin{theorem}\label{primary-decomp}{\rm \cite[Theorem 3.2]{HHHKR}}
Let $G$ be a simple graph on  $[n]$. Then $J_{G}=\bigcap_{\mathcal{S}\subseteq[n]}P_{\mathcal{S}}(G)$.
\end{theorem}

\begin{remark}\label{Knprime} Since $P_{\emptyset}(G)$ does not contain monomials, it follows  that $P_{\mathcal{S}}(G)\not\subseteq P_{\emptyset}(G) $ for every $\mathcal{S}$. Thus, by Theorem \ref{primary-decomp},  $P_{\emptyset}(G)$ is a minimal prime of $J_{G}$. In particular, when $G$ is connected, we have that $J_{\mathcal{K}_n}$ is a minimal prime.
\end{remark}

\begin{definition} A vertex $i$ of $G$ is a {\it cut-point} of $G$ if $G$ has less connected components than $G_{[n]\setminus \lbrace i\rbrace}$.
\end{definition}
As the following result shows, among all $P_{\mathcal{S}}(G)$'s, the minimal primes of $J_G$ are those corresponding to particular subsets of the vertices. 
\begin{theorem}{\rm \cite[Corollary 3.9]{HHHKR}}\label{minimal_primes}
Let $G$ be a connected graph on $[n]$, and $\mathcal{S}\subseteq[n]$. Then $P_{\mathcal{S}}(G)$ is a minimal prime ideal of $J_{G}$ if and only if $\mathcal{S}=\emptyset$, or $\mathcal{S}\neq\emptyset$ and each $i\in\mathcal{S}$ is a cut point of the graph $G_{([n]\setminus\mathcal{S})\cup\lbrace i\rbrace}$.
\end{theorem}
In case the graph is not connected, the minimal primes of its binomial edge ideal are given  by the sums of the minimal primes of the binomial edge ideals of its connected components.
\begin{proposition}{\rm \cite[Problem 7.8]{HHO}}\label{minimalpri}
Let $G$ be a graph on $[n]$ with connected components $G_{1},\dots,G_{r}$, and let $T\subset[n]$. Set $T_{i}=T\cap V(G_{i})$ for $i=1,\dots,r$. Then 
\begin{itemize}
\item[(a)] $P_{T}(G)=\sum_{i=1}^{r}P_{T_{i}}(G_{i})S$
\item[(b)] $P_{T}(G)$ is a minimal prime ideal of $J_{G}$ if and only if each $P_{T_{i}}(G_{i})$ is a minimal prime ideal of $J_{G_{i}}$.
\end{itemize}
\end{proposition}

%

In this work we will be interested in studying closed graphs. This class of graphs has been algebraically characterized in terms of Gr\"obner bases of their binomial edge ideals. However, it is worth mentioning that they were already widely studied in combinatorics, where they are known as \emph{unit-interval graphs}. 

\begin{definition}\label{closed-graph}{\rm \cite[Theorem 7.2]{HHO}}
Let $G$ be a simple graph on $[n]$, and let $\prec$ be the lexicographic order on $S$ induced by $x_{1}\succ\dots\succ x_{n}\succ y_{1}\succ\dots\succ y_{n}$. The followings are equivalent:
\begin{itemize}
\item[(a)] The natural generators $f_{ij}$ of $J_{G}$ form a quadratic Gr\"obner basis;

\item[(b)] For all integers $1\leq i<j<k\leq n$, if $\lbrace i, k\rbrace \in E(G)$ then $\lbrace i, j\rbrace\in E(G)$ and $\lbrace j, k\rbrace\in E(G)$.
\end{itemize}
A graph $G$ is \emph{closed} if there exists a labelling of its vertices such that one of the previous equivalent conditions is satisfied. 
\end{definition}

We point out that both the conditions in Definition \ref{closed-graph} do not only depend on the isomorphism type of the graph, but also on the labelling of its vertices. When we say that a graph $G$ is closed, we assume that $G$ is closed with respect to the given labelling of its vertices.\par 

Among all the algebraic invariants of binomial edge ideals, Castelnuovo-Mumford regularity has been intensively investigated and several conjectures have been settled (see e.g., {\rm \cite{Via-Zar}, \cite{Jay-Nara}, \cite{Kiani-Mada}, \cite{Kiani-Mada2}, \cite{Mat-Mur},\cite{ Mala-Mada}, \cite{Sche-Zafar}}).\par 
 Below we recall the definition of Castelnuovo-Mumford regularity and the main known results on the regularity of binomial edge ideals.
\begin{definition}{\rm \cite{Eise-Sys}}
Let $I\subset S$ be a graded ideal, and let ${\bf F}$ be the minimal  graded resolution  of $S/I$  as an $S$-module: 
$$
\xymatrix{
{\bf F}:0\ar[r] &\bigoplus_{j}S(-j)^{b_{g,j}}\ar[r] &\cdots\ar[r] &\bigoplus_{j}S(-j)^{b_{1,j}}\ar[r] &S\ar[r] &S/I
},
$$  
The Castelnuovo-Mumford  regularity  of $S/I$ is defined as:
$$
{\rm reg}(S/I):=\max\lbrace j-i\:|\: b_{i,j}\neq 0\rbrace.
$$
\end{definition}

\begin{remark}\label{addi_reg}
Let $G$ be  a graph  with connected components $G_{1},\dots, G_{r}$ and set $S_{i}=K[\lbrace x_{j}, y_{j}\rbrace_{j\in V(G_{i})}]$. Then 
$$
{\rm reg}(S/J_{G})=\sum_{i=1}^{r}{\rm reg}(S_{i}/J_{G_{i}}).
$$
\end{remark}

\begin{proposition}\cite[Theorem 1.1]{Mat-Mur},\cite[Theorem 3.2]{Via-Zar}\label{regBEI}
Let $G$ be a simple graph on $[n]$ and $\ell$ be the length of its longest induced path. Then
$$\ell +1 \leq \reg (J_G) \leq n.$$
In addition, if $G$ is closed, $\reg (J_G)=\ell +1.$
\end{proposition}


\subsection{Edge ideals and v-numbers}
Given a simple graph $G$ with vertices $V(G)=[n]$, the \emph{edge ideal of } $G$ is a squarefree monomial ideal defined by 
$$
I(G):=\lbrace t_{i}t_{j}\:|\: \lbrace i,j\rbrace\in E(G)\rbrace\subseteq K[t_{1},\dots,t_{n}].
$$

This ideal was introduced in {\rm \cite{RCOMGR}}, and has been studied in the literature from different perspectives; see {\rm \cite{CHHVTuyl,FASKTN,HZXZ,SIMVASVI,mon-alg}} and the references therein.

\begin{remark} By Definition \ref{closed-graph}, if $G$ is a closed graph then the initial ideal of $J_G$ (with respect to the lexicographic order) is given by $\init(J_G)=\{x_i y_j \mid \{i,j\} \in E(G), i<j\}$. Hence, $\init(J_G)$ can be seen as the edge ideal of a bipartite graph $H_G$, as shown in the next example. We will call $H_G$ the \emph{initial graph} of $G$. Moreover, the length of the longest induced path of $G$ is known to be equal to the induced matching number of $H_G$ (see \cite[Proposition 7.32]{HHO}).
\end{remark}

\begin{example}\label{close_graph_corre}

Let $G$ be the following closed graph on $6$ vertices.

\begin{figure}[H]
\centering
\begin{tikzpicture}[line width=.5pt,scale=0.75]
		\tikzstyle{every node}=[inner sep=1pt, minimum width=5.5pt] 
\tiny{
\node (1) at (.5,1){$\bullet$};
\node (2) at (1,0) {$\bullet$};
\node (3) at (.5,-1) {$\bullet$};
\node (4) at (-.5,-1) {$\bullet$};
\node (5) at (-1,0){$\bullet$};
\node (6) at (-.5,1) {$\bullet$};
\node at (.7,1){$1$};
\node at (1.2,0) {$2$};
\node at  (.5,-1.3){$3$};
\node at (-.5,-1.3){$4$};
\node at (-1.2,0){$5$};
\node at (-.7,1) {$6$};
\draw[-,line width=1pt] (1) to (2);
\draw[-,line width=1pt] (2) to (3);
\draw[-,line width=1pt] (3) -- (4);
\draw[-,line width=1pt] (4) -- (5);
\draw[-,line width=1pt] (5) -- (6);
\draw[-,line width=1pt] (4) -- (5);
\draw[-,line width=1pt] (5) -- (6);
\draw[-,line width=1pt] (1) -- (3);
\draw[-,line width=1pt] (4) -- (6);
}
\end{tikzpicture}
\caption{}\label{correspondance}
\end{figure}

Then, the initial ideal of $J_{G}$ is given by 
$
\init(J_G)=\{x_1 y_2,x_2 y_3,x_3 y_4,x_4 y_5,x_5 y_6,x_1 y_3,x_4 y_6\}
$
and the {\it initial graph} $H_G$ is the following bipartite graph
\begin{figure}[H]
\centering
\begin{tikzpicture}[line width=.5pt,scale=0.75]
		\tikzstyle{every node}=[inner sep=1pt, minimum width=5.5pt] 
\tiny{
\node (1) at (-2.5,1){$\bullet$};
\node (2) at (-1.5,1) {$\bullet$};
\node (3) at (-.5,1) {$\bullet$};
\node (4) at (.5,1) {$\bullet$};
\node (5) at (1.5,1){$\bullet$};
\node (6) at (2.5,1) {$\bullet$};
\node (7) at (-2.5,-1){$\bullet$};
\node (8) at (-1.5,-1) {$\bullet$};
\node (9) at (-.5,-1) {$\bullet$};
\node (10) at (.5,-1) {$\bullet$};
\node (11) at (1.5,-1){$\bullet$};
\node (12) at (2.5,-1) {$\bullet$};
\node at (-2.5,1.3){$x_1$};
\node at (-1.5,1.3) {$x_2$};
\node at  (-.5,1.3){$x_3$};
\node at (.5,1.3){$x_4$};
\node at (1.5,1.3){$x_5$};
\node at (2.5,1.3) {$x_6$};
\node at (-2.5,-1.3){$y_1$};
\node at (-1.5,-1.3) {$y_2$};
\node at  (-.5,-1.3){$y_3$};
\node at (.5,-1.3){$y_4$};
\node at (1.5,-1.3){$y_5$};
\node at (2.5,-1.3) {$y_6$};
\draw[-,line width=1pt] (1) to (8);
\draw[-,line width=1pt] (2) to (9);
\draw[-,line width=1pt] (3) -- (10);
\draw[-,line width=1pt] (4) -- (11);
\draw[-,line width=1pt] (5) -- (12);
\draw[-,line width=1pt] (1) -- (9);
\draw[-,line width=1pt] (4) -- (12);
}
\end{tikzpicture}
\caption{}\label{correspondance_bipar}
\end{figure}

A longest induced path in $G$ is given by $1,3,4,6$ and the corresponding induced matching in $H_G$ is $\{x_1,y_3\},\{x_3,y_4\},\{x_4,y_6\}$. So, $\ell_G=\IM(H_G)=3$. 
\end{example}

In \cite{DR}, the first author and Villarreal investigate the v-number of edge ideals of clutters and find a combinatorial description for it in terms of particular stable sets.
Before stating the result, we recall some preliminary notions and notations.
\begin{terminology}Let $G$ be a simple graph.
\begin{itemize} [itemsep=1.3pt,topsep=2pt,leftmargin=0.3in] 
\item A subset $\mathcal{V}\subseteq V(G)$ is a {\it vertex cover} if every edge of $G$ is incident with at least one vertex in $\mathcal{V}$. A vertex cover $\mathcal{V}$ is {\it minimal} if each proper subset of $\mathcal{V}$ is not  a vertex cover of $G$. 
\item A subset $\mathcal{V}\subseteq V(G)$ is called {\it stable} (or {\it independent}) if no two vertices in $\mathcal{V}$ are joined by an edge of $G$. Note that a set of vertices $\mathcal{V}$ is a (maximal) stable set if and only if $V(G)\setminus \mathcal{V}$ is a (minimal) vertex cover.
\item The neighbourhood of a vertex $v \in V(G)$ is the set $N_G(v):=\{i\in V(G) \mid \{v,i\} \in E(G)\}$. Analogously, if $\mathcal{V} \subseteq V(G)$, then $N_{G}(\mathcal{V}):= \{k\in V(G) \mid \{i,k\} \in E(G) \text{ for some } i \in \mathcal{V}\}$.
\end{itemize}
\end{terminology} 

\begin{theorem}{\rm \cite[Theorem 3.5]{DR}}\label{v-number-clutters-graphs}
Let $G$ be a connected graph.
If $I(G)$ is not prime, then  
$$\mathrm{v}(I(G))=\min\{|A|\colon 
A\in\mathcal{A}_{G}\},
$$
where $\mathcal{A}_{G}$ denotes the family of all stable sets $A$ of $G$ whose set of neighbours $N_{G}(A)$ is a minimal vertex cover of $G$. 
\end{theorem}

\begin{remark} It follows from the above theorem that $\mathrm{v}(I(G))\leq i(G)$, where $i(G)$ is the \emph{independent domination number}, that is the minimum size of an independent set such that every vertex not in the set is adjacent to a vertex in it (see {\rm \cite[Corollary 3.6]{DR}}). Again, this stresses the interconnected nature of v-numbers and  domination theory (see also {\rm\cite[Theorem 9]{Civan}}).
\end{remark}

%
If $I$ is an ideal without embedded primes (e.g., edge ideals, binomial edge ideals), we have an alternate description for the v-number using initial degrees of certain modules. This characterization also allows us to compute the v-number using \emph{Macaulay2} (see Code \ref{v-number-procedure}). 

\begin{proposition}\cite[Theorem 10 (d)]{Gri-Re-Vill} \label{noembedded}
Let $I$ be a graded ideal with no embedded primes, and $\mathfrak{p}\in {\rm Ass}(I)$. Then 
$$
{\rm v}_{\mathfrak{p}}(I)=\alpha((I:\mathfrak{p})/I)$$
where, for any graded module $M \neq 0$, $\alpha(M):=\min \{\deg(f)\mid f\in M\setminus\{0\}\}$ and, by convention, we set $\alpha(0)=0$ for $M=0$. 
In particular, $
{\rm v}(I)=\min\{\alpha((I:\mathfrak{p})/I)\mid \mathfrak{p}\in{\rm Ass}(I)\}.$
\end{proposition}

\section{v-numbers of binomial edge ideals}\label{v-num-bino-section}


This section is concerned with the study of v-numbers of binomial edge ideals. Relying on some results in \cite{Fatemeh-Leila}, we prove that the algebraic invariant ${\rm v}_{J_{\mathcal{K}_{n}}}(J_{G})$ coincides with a purely combinatorial invariant of the graph $G$, namely its connected domination number (see Theorem \ref{comb_at_comple}). We then use this result to relate ${\rm v}(J_{G})$ and ${\rm v}({\rm in}_{\prec}(J_{G}))$ in the case of closed graphs (see Theorem \ref{relation_initial}).

\subsection{v-numbers and connected domination}
We start by proving that the v-number of binomial edge ideals is additive on connected components and we can then restrict ourselves to connected graphs.

\begin{proposition}\label{additivity}
Let $G$ be a graph on $[n]$ and $G_{1},\dots,G_{r}$ be its connected components. Let $S_{i}$ be the polynomial ring in the indeterminates indexed by the vertices of $G_{i}$. Then  
$$
{\rm v}(J_{G})={\rm v}(J_{G_{1}}S+\dots+J_{G_{r}}S)={\rm v}(J_{G_{1}})+\dots+{\rm v}(J_{G_{r}}).
$$
\end{proposition}
\begin{proof}
Since $G_{1},\dots,G_{r}$ are the connected components of $G$, we can write $J_{G}=J_{G_{1}}S+\dots +J_{G_{r}}S$. Let  $T\subseteq [n]$ such that ${\rm v}(J_{G})={\rm v}_{P_T(G)}(J_{G})$. Then, there exists $f\in S$ with $\deg (f)={\rm v}(J_{G})$ and 
$$
(J_{G}:f)=P_{T}(G)=\sum_{i=1}^{r}P_{T_{i}}(G_{i})S\quad \hbox{(by Proposition \ref{minimalpri}).}
$$
In particular, this implies that $f\in(J_{G}:P_T(G)) \setminus J_G$. Moreover,
$$
(J_{G}:P_T(G))=\bigcap_{i=1}^{r}(J_{G}:P_{T_i}(G_{i}))=\bigcap_{i=1}^r(J_{G_{i}}:P_{T_i}(G_{i}))+ J_G.
$$
Since $\bigcap_{i=1}^r(J_{G_{i}}:P_{T_i}(G_{i}))=\prod_{i=1}^r(J_{G_{i}}:P_{T_i}(G_{i}))$ {\rm \cite[Section 4]{SIMVASVI}}, we can write 
$$f=f_{1}\cdots f_{r} \quad \hbox{with } f_{i}\in (J_{G_{i}}:P_{T_i}(G_{i})) \setminus J_{G_{i}} .$$  Thus $ (J_{G_{i}}:f_i)=P_{T_i}(G_{i})$ and 
$${\rm v}(J_{G_{1}})+\dots+{\rm v}(J_{G_{r}})\leq \deg(f_{1})+\dots+\deg(f_{r})={\rm v}(J_{G}).$$
As for the other inequality, for every connected component of $G$ we choose $f_{i}\in S_{i}$  such that ${\rm v}(J_{G_{i}})=\deg(f_{i})$. Hence, $ f_{i} \in (J_{G_{i}}:P_{T_{i}}(G_{i})) \setminus  J_{G_{i}}$ for some $P_{T_{i}}(G_{i})\in{\rm Min}(S_{i}/J_{G_{i}})$. As before, we have
$$f:=f_{1}\cdots f_{r} \in \bigcap_{i=1}^r(J_{G_{i}}:P_{T_i}(G_{i}))= (J_G: P_T (G))/J_G, $$
where $P_{T}(G)=\sum_{i=1}^{r}P_{T_{i}}(G_{i})S$ and, by Lemma \ref{minimalpri}, it is a minimal prime of $J_G$. In particular, using Proposition \ref{noembedded}, we get
$$
{\rm v}(J_{G})\leq \alpha((J_G:P_T (G))/I) \leq \deg (f) = \deg(f_{1})+\cdots + \deg(f_{r})={\rm v}(J_{G_{1}})+\dots+{\rm v}(J_{G_{r}})
$$
and the proof is complete.
\end{proof}

From here on, we can then limit our study to the case of connected graphs. By Remark \ref{Knprime}, if $G$ is connected, $J_{\mathcal{K}_n}$ is a minimal prime of $J_G$ and ${\rm v}(J_{G}) \leq {\rm v}_{J_{\mathcal{K}_n}}(J_{G})$. \par 
 We then come to the main result of this subsection. The following shows that ${\rm v}_{J_{\mathcal{K}_n}}(J_{G})$ coincides with a combinatorial invariant of the graph, namely its connected domination number.  

\begin{theorem}\label{comb_at_comple}
Let $G$ be a connected graph with $n$ vertices and $J_{G}$ be its
 binomial edge ideal. Then, the localization of the {\rm v}-number of $J_G$
 at $J_{\mathcal{K}_{n}}$ is given by
 $$
{\rm v}_{J_{\mathcal{K}_{n}}}(J_{G})=\begin{cases}\min\lbrace |B|\: |\:
B\in\mathcal{D}_{c}(G)\rbrace=\gamma_c (G) &\quad \mbox{if }G\neq \mathcal{K}_n\\
0 & \quad \hbox{if }G=\mathcal{K}_n
\end{cases}
$$
where $\mathcal{D}_{c}(G)= \left\lbrace B \subseteq [n] \mid \forall \lbrace i,j\rbrace\not\in E(G), \exists \ \textit{a path} \  P: i,i_{1},\dots,i_{s},j \text{ s.t. }\lbrace i_{1},\dots,i_{s}\rbrace\subseteq B\right\rbrace$.
\end{theorem}

To prove Theorem \ref{comb_at_comple} we will use the equivalent characterization of v-numbers of graded ideals with no embedded primes, given in Proposition \ref{noembedded}.\par
We are interested in the case $\mathfrak{p}=J_{\mathcal{K}_n}$, so we want to compute $\alpha((J_{G}:J_{\mathcal{K}_{n}})/J_G)$. First, we observe that $(J_{G}:J_{\mathcal{K}_{n}})=\bigcap_{e\not\in E(G)}(J_{G}:f_{e})$ and the colon ideals appearing in this intersection were already studied by Mohammadi and Sharifan in \cite{Fatemeh-Leila}. 

\begin{definition}{\rm \cite[Definition 3.1]{Fatemeh-Leila}}\label{e-graph}
Let $G$ be a simple graph on $[n]$. Given $e=\lbrace i,j\rbrace\not\in E(G)$, we define $G_{e}$ to be the graph on $[n]$ with edges:
\begin{align*}
E(G_{e})=E(G)\cup\lbrace \lbrace k,l\rbrace\:|\: k,l\in N_G(i)\hbox{ or } k,l\in N_G(j)\rbrace.
\end{align*}
\end{definition}

\begin{theorem}{\rm \cite[Theorem 3.7]{Fatemeh-Leila}}{\rm \cite[Lemma 3.8]{Fatemeh-Leila}}\label{fat-lei}
Let $G$ be a graph on $[n]$ and $e=\lbrace i,j\rbrace\not\in E(G)$. Then 
$$
(J_{G}:f_{e})=J_{G_{e}}+(g_{P,t}\:|\: P:i,i_{1},\dots,i_{s},j \hbox{ is a path from }i \hbox{ to }j\hbox{ and }0\leq t\leq s),
$$
where $g_{P,0}:=x_{i_{1}}\cdots x_{i_{s}}$ and for each $1\leq t\leq s$, $g_{P,t}:=y_{i_{1}}\cdots y_{i_{t}}x_{i_{t+1}}\cdots x_{i_{s}}$.\par 
Moreover, if  $P: i,i_{1},\dots,i_{s},j$ is a path from $i$ to $j$ and $C,D$ are two arbitrary subsets of $\{i_{1},\dots,i_{s}\}$ with $ C\cap D = \emptyset$ and $C\cup D=\{i_{1},\dots,i_{s}\}$, then $$h_{C,D}:=\prod_{k \in C} x_{k} \prod_{k\in D} y_k \in (J_{G}:f_{e}).$$
\end{theorem}

In order to extend Theorem \ref{fat-lei} and compute $(J_{G}:J_{\mathcal{K}_{n}})$, we need the following preliminary observation.
\begin{lemma}\label{grobner-basis}
Let $G$ be a graph on $n$ vertices. With the notation of Proposition \ref{comb-des-colo}, the set
$$
\mathcal{G}=\bigcup_{i<j}\{u_{\pi}f_{ij}\mid\pi \hbox{ is an admissible path from }i\hbox{ to }j\}\cup \{g_{C,D}\mid C \cap D= \emptyset, C \cup D\in\mathcal{D}_{c}(G)\}
$$
is a Gr\"obner basis of $J_{G}+(g_{C,D}\mid C \cap D= \emptyset, C \cup D\in\mathcal{D}_{c}(G))$ with respect to the lexicographic order. \par 
In particular, 
$$\init(J_{G}+(g_{C,D}\mid C \cap D= \emptyset, C \cup D\in\mathcal{D}_{c}(G)))= \init(J_{G})+\init((g_{C,D}\mid C \cap D= \emptyset, C \cup D\in\mathcal{D}_{c}(G))).$$
\end{lemma}
\begin{proof}
In order to prove that $\mathcal{G}$ is a Gr\"obner basis we apply Buchberger's criterion and we show that all $S$-pairs $S(f,h)$ with $f,h\in \mathcal{G}$ reduce to zero. To this purpose, we recall that $\mathcal{G}'=\bigcup_{i<j}\{u_{\pi}f_{ij}\mid\pi \hbox{ is an admissible path from }i\hbox{ to }j\}$ is a Gr\"obner basis of $J_G$ (we refer the reader to {\rm \cite[Theorem 2.1]{HHHKR}} for more about admissible paths and for a proof of this result). As a consequence, we have that $S(f,h)$ reduces to zero whenever $f, h \in \mathcal{G}'$. Moreover, the $S$-polynomial of two monomials always reduces to zero. So, we are left with verifying that $S(f,h)$ reduces to zero when
$f=u_{\pi}f_{ij} \in \mathcal{G}$ and $h=g_{C,D}$. In this case, we have 
\begin{align*}
S(f,g)&=\cfrac{{\rm lcm}(u_{\pi}x_{i}y_{j},g_{C,D})}{u_{\pi}x_{i}y_{j}}u_{\pi}(x_{i}y_{j}-x_{j}y_{i})-\cfrac{{\rm lcm}(u_{\pi}x_{i}y_{j},g_{C,D})}{g_{C,D}}g_{C,D}\\
&=\cfrac{{\rm lcm}(u_{\pi}x_{i}y_{j},g_{C,D})}{x_{i}y_{j}}(x_{i}y_{j}-x_{j}y_{i})-{\rm lcm}(u_{\pi}x_{i}y_{j},g_{C,D})\\
&=-\cfrac{{\rm lcm}(u_{\pi}x_{i}y_{j},g_{C,D})}{x_{i}y_{j}}x_{j}y_{i}.
\end{align*}
Notice that the monomial $-\cfrac{{\rm lcm}(u_{\pi}x_{i}y_{j},g_{C,D})}{x_{i}y_{j}}x_{j}y_{i}$ is divisible by either $g_{C,D}$ or $g_{C',D'}$, where $C'=(C\setminus\{x_{i}\})\cup\{x_{j}\}$ and $D'=(D\setminus\{y_{j}\})\cup\{y_{i}\}$. Thus, $S(f,g)$ reduces to zero.
\end{proof}

The following proposition extends Mohammadi and Sharifan's result and it is key in the proof of Theorem \ref{comb_at_comple}.
 
\begin{proposition}\label{comb-des-colo} With the above notation, let $G\neq \mathcal{K}_n$ be a graph on $[n]$
$$
(J_{G}:J_{\mathcal{K}_{n}})=\bigcap_{e\not\in E(G)}(J_{G}:f_{e})=J_{G}+(g_{C,D}\mid C \cup D \in\mathcal{D}_{c}(G), C \cap D= \emptyset)
$$
where $g_{C,D}:=\prod_{k \in C} x_{k} \prod_{k\in D} y_k$.
\end{proposition}

\begin{proof}
Let $g_{C,D}=\prod_{k\in C}x_{k}\prod_{k\in D}y_{k}.$ be a monomial with $B:=C\cup D=\lbrace i_{1},\dots,i_{s}\rbrace\in \mathcal{D}_{c}(G)$ and $C\cap D=\emptyset$. We want to prove that $g_{C,D}\in (J_{G}:f_{e})$ for every $e=\lbrace i, j\rbrace\not\in E(G)$. By definition of $\mathcal{D}_{c}(G)$, for every $e=\{i,j\} \notin E(G)$ we can find a path $P_{e}:i, j_{1},\dots,j_{s},j$ from $i$ to $j$ such that $\mathcal{J}:=\lbrace j_{1},\dots,j_{s}\rbrace\subseteq B$. We set $C':=\mathcal{J}\cap C$ and $D':=\mathcal{J}\cap D$ and we define $$q:=\prod_{k\in C'}x_{k}\prod_{k\in D'}y_{k}.$$ We observe that $\mathcal{J}=C'\cup D'$ and $C'\cap D'=\emptyset$. So, by Theorem \ref{fat-lei}
$q\in (J_{G}:f_{e}).$ Moreover, by construction, $q|g_{C,D}$. Thus, $g_{C,D} \in (J_{G}:f_{e})$. This proves that 
$$
(g_{C,D}\mid C \cap D= \emptyset, C \cup D=\lbrace i_{1},\dots,i_{s}\rbrace\in\mathcal{D}_{c}(G))\subseteq \bigcap_{e\not\in E(G)}(J_{G}:f_{e})=(J_{G}:J_{\mathcal{K}_{n}}),
$$
and since $J_{G}\subseteq(J_{G}:J_{\mathcal{K}_{n}})$ we can conclude that 
$$
J_G+(g_{C,D}\mid C \cap D= \emptyset, C \cup D=\lbrace i_{1},\dots,i_{s}\rbrace\in\mathcal{D}_{c}(G))\subseteq \bigcap_{e\not\in E(G)}(J_{G}:f_{e})=(J_{G}:J_{\mathcal{K}_{n}}).$$

As for the other containment, we use induction on the number of vertices. If $n\leq 4$ the result can be easily checked. So, let us assume by induction hypothesis that the result holds for each graph with fewer  number of vertices than $G$. Consider an element $f\in(J_{G}:J_{\mathcal{K}_{n}})$ and let $\mathcal{D}_{c}(G)=\{A,B_{1},\dots,B_{r}\}$. Define $I:=J_{G}+(g_{C,D}\mid C \cap D= \emptyset, C \cup D\in\mathcal{D}_{c}(G)\setminus \{A\})$. Then, by division algorithm, we can write $f=f'+f''$ such that $f''\in I$ and no monomial of $f'$ lies in $\init(I)$. Using the first part of the proof, $f'=f-f''\in(J_{G}:J_{\mathcal{K}_{n}})$, and we have that
\begin{equation}\label{colon}
f'f_{e}= h_{1e}f_{e_{1}}+\dots+h_{ke}f_{e_{k}},\ \forall e\not\in E(G)
\end{equation}
where $f_{e_1}, \ldots, f_{e_k}$ are the natural generators of $J_G$.
Let $g:={\rm in}_{\prec}(f')$ and $A=\{i_{1},\dots,i_{s}\}$. We show that for all $i_c\in A$ either $x_{i_{c}}|g$ or $y_{i_{c}}|g$. Assume by contradiction that $x_{i_{c}}\not| g$ and $y_{i_{c}}\not| g$ for some $1\leq c\leq s$ and set $x_{i_{c}}=y_{i_{c}}=0$ in Eq. \ref{colon}. Define $w:=f'|_{x_{i_{c}}=y_{i_{c}}=0}$, and $\hat{h}_{je}:=h_{je}|_{x_{i_{c}}=y_{i_{c}}=0}$. Then, 
\begin{equation}
wf_{e}=\sum_{i_{c}\not\in e_{j}} \hat{h}_{je}f_{e_{j}} \quad \forall e\not\in E(G), \ i_c \notin e.
\end{equation}
Therefore, $w\in (J_{G_{1}}:J_{\mathcal{K}_{n-1}})$, where $G_{1}$ is the induced subgraph of $G$ on the vertices $V(G)\setminus \{i_{c}\}$. Notice that, since $x_{i_{c}}\not| g$ and $y_{i_{c}}\not| g$, we have ${\rm in}_{\prec}(w)=g$. By induction hypothesis 
$$(J_{G_{1}}:J_{\mathcal{K}_{n-1}})= J_{G_1}+(g_{C,D}\mid C \cap D= \emptyset, C \cup D\in\mathcal{D}_{c}(G_1))$$
and Lemma \ref{grobner-basis} implies that either $u_{\pi}x_i y_j|g$ where $i<j$ and $u_{\pi}$ is an admissible path in $G_{1}$ or $g_{\hat{C},\hat{D}}|g$, where $\hat{C}\cup\hat{D}\in\mathcal{D}_{c}(G_1)$. \par 
In the first case, since $J_{G_{1}}\subseteq J_{G}$, we would have that $g \in \init(J_G) \subseteq\init(I)$, a contradiction. \par 
In the second case, we show that it is always possible to find a monomial $g_{C',D'} \in (g_{C,D}\mid C \cap D= \emptyset, C \cup D\in\mathcal{D}_{c}(G)\setminus \{A\})$ such that $g_{C',D'}|g$, and we again obtain a contradiction. By hypothesis, we know that $g_{\hat{C},\hat{D}}|g$ where $\hat{C}\cup\hat{D}\in\mathcal{D}_{c}(G_{1})$. If $\hat{C}\cup\hat{D}\in\mathcal{D}_{c}(G) \setminus \{A\}$ there is nothing to prove. So, let us assume that $\hat{C}\cup\hat{D}\notin\mathcal{D}_{c}(G) \setminus \{A\}$  To do so, we begin by observing that if $\mathcal{V} \in \mathcal{D}_c(G_1)$ and $j_c \in N(i_c)$, then $\mathcal{V}\cup \{j_c\} \in \mathcal{D}_c(G)$. Then, since by hypothesis $g_{\hat{C},\hat{D}}|g$ where $\hat{C}\cup\hat{D}\in\mathcal{D}_c(G_{1})$, it is sufficient to find a neighbour of $i_{c}$, say $j_c$, such that either $x_{j_c} \in {\rm Supp}(g)$ or $y_{j_c} \in {\rm Supp}(g)$. Let $j_{c}\in N_{G}(i_{c})$. If $x_{j_{c}}|g$ or $y_{j_{c}}|g$, we are done. Otherwise, applying the same argument as before, we conclude that
$$
w'=f'|_{x_{j_{c}}=y_{j_{c}}=0}\in (J_{G_{2}}:J_{\mathcal{K}_{n-1}}),
$$
where $G_{2}$ is the induced subgraph of $G$ on the vertices $V(G)\setminus \{j_{c}\}$. If $G_{2}=\mathcal{K}_{n-2}$, then the vertex $i_{c}$ has degree $n-1$ and this implies that $\hat{C}\cup\hat{D}$ must contain a neighbour of $i_{c}$, and we are done. If $G_{2}\neq\mathcal{K}_{n-2}$, again, by induction hypothesis and Lemma \ref{grobner-basis}, there exists a monomial $g_{\hat{F},\hat{H}}|{\rm in}_{\prec}(w')=g$ with $\hat{F}\cup\hat{H}\in\mathcal{D}_c(G_{2})$. This means that for every edge $e=\{i_{c}, v\}\not\in E(G_2)$ there is a path from $i_{c}$ to $v$ contained in $\hat{F}\cup\hat{H}$. In particular, $\hat{F}\cup\hat{H}$ contains a neighbour of $i_{c}$, say $v_c$.  Since, $g_{\hat{F},\hat{H}}|g$, we can conclude that ${\rm Supp}(g)$ contains either $x_{v_c}$ or $y_{v_c}$. As observed before,  $\hat{C}\cup\hat{D}\cup\{v_{c}\}$ is in $\mathcal{D}_{c}(G)$. Therefore the monomial $g$ is divided by an element of $I$, a contradiction. 
This proves that for all $i_c\in A$, either $x_{i_{c}}$ or $y_{i_{c}}$ divides $g$. We conclude that $g\in (g_{C,D}\:|\: C\cap D=\emptyset,\: C\cup D=A \in \mathcal{D}_c (G))$. Now, applying the same argument to $f'-g$, we get that $f'\in (g_{C,D}\:|\: C\cap D=\emptyset,\: C\cup D=A \in \mathcal{D}_c (G))$ and this completes the proof. 
\end{proof}

Using the above result, Theorem \ref{comb_at_comple} follows immediately from Proposition \ref{noembedded}.

\begin{proof}[Proof of Theorem \ref{comb_at_comple}]
Assume that $G \neq \mathcal{K}_n$. Since the ideal $J_{G}$ is radical, Proposition \ref{noembedded}  implies that 
$$
{\rm v}_{J_{\mathcal{K}_{n}}}(J_{G})=\alpha((J_{G}:J_{\mathcal{K}_{n}})/J_{G}).
$$
By Proposition \ref{comb-des-colo}, we get 
\begin{align*}
\begin{split}
\alpha((J_{G}:J_{\mathcal{K}_{n}})/J_{G})&=\alpha\left((J_{G}+(g_{C,D} \mid C \cap D= \emptyset, C \cup D\in\mathcal{D}_{c}(G)))/J_{G}\right)=\\
&=\alpha\left( (g_{C,D}\mid C \cap D= \emptyset, C \cup D\in\mathcal{D}_{c}(G))\right).
\end{split}
\end{align*}
Therefore ${\rm v}_{J_{\mathcal{K}_{n}}}(J_{G})=\min\lbrace |A|\: |\: A\in\mathcal{D}_{c}(G)\rbrace$. Notice that $\mathcal{D}_{c}(G)$ is the set of all connected dominating sets of $G$. Indeed, let $B$ be a CDS and let  $\{i,j\} \notin E(G)$.  By definition of CDS, either $i, j \in B$ or there exist $b_i, b_j \in B$ such that $\{i, b_i\}, \{j,b_j\} \in E(G)$. Thus, since $G_B$ is connected, we can always find a path $P: i, b_i,\ldots,b_j, j$ in $G$ from $i$ to $j$ such that $\{b_i,\ldots,b_j \} \subseteq B$. 
Viceversa, let $B \in \mathcal{D}_{c}(G)$. By definition of $\mathcal{D}_{c}(G)$, for every $a \in [n] \setminus B$ there exists $b_a \in B$ such that $\{b_a,a\} \in E(G)$. Moreover, it is straightforward to see that $G_B$ is connected. So, $B$ is a CDS. Thus, 
$$
\min\lbrace |B|\: |\: B\in\mathcal{D}_{c}(G)\rbrace = \gamma_c (G).
$$
If  $G=\mathcal{K}_n$ is the complete graph, then $
{\rm v}(J_{\mathcal{K}_{n}})={\rm v}_{J_{\mathcal{K}_{n}}}(J_{\mathcal{K}_{n}})=0.$
\end{proof}

%
%

As a consequence of Theorem \ref{comb_at_comple}, we obtain algebraic bounds for the connected domination number of a graph. In particular, we also get a sharp combinatorial upper bound for the v-number of any binomial edge ideal.
\begin{corollary}\label{corollaryMain}
Let $G$ be a connected graph. Then
$${\rm v}(J_{G})\leq \gamma_c (G)\leq {\rm reg}(S/T_G)$$
where $T_G$ is a spanning tree of $G$ with the largest number of leaves.
\end{corollary}

\begin{proof}
The statement of Theorem \ref{comb_at_comple} can be rephrased in terms of spanning trees. As explained in the introduction, $\gamma_c(G)=n-{\rm lf}_{\max}(G)$ where ${\rm lf}_{\max}(G)$ is the maximum leaf number of $G$. Therefore, by \cite[Theorem 4.1]{Jay-Nara},
$${\rm v}(J_{G})\leq {\rm v}_{J_{\mathcal{K}_{n}}}(J_{G})=\gamma_c(G)=n-{\rm lf}_{\max}(G) \leq {\rm reg}(S/T_G)$$
where $T_G$ is a spanning tree of $G$ with the largest number of leaves (equivalently, with $\gamma_c (G)$ internal vertices).
\end{proof}

\begin{remark} We remark that the above corollary relates the v-number of $J_G$ with the regularity of the binomial edge ideal of a particular subgraph $T$ of $G$. This subgraph is in general not induced. However when $T$ is induced, ${\rm reg}(S/T_G) \leq {\rm reg}(S/J_G)$ and so $ {\rm v}(J_{G}) \leq {\rm reg}(S/J_G)$. See also Example \ref{ExampleNC} and Question \ref{question2}. 
\end{remark} 

Exploiting Proposition \ref{noembedded} and Theorem \ref{comb_at_comple}, at the end of this paper we provide a \emph{Macaulay2} code to compute the minimum connected domination number $\gamma_c(G)$ of any connected graph $G$ (see Macaulay 2 Code \ref{gammafunction}).

 \subsection{v-numbers of closed graphs} \label{sectionclosed}
We know from Section \ref{preli-section} that closed graphs are characterized as those graphs whose binomial edge ideals have a quadratic Gr\"obner basis given by the natural generators. In particular, the initial ideal is squarefree and can be also understood as the edge ideal of a bipartite graph.
In the following, we establish a relation between the v-number of $J_G$ and that of its initial ideal when $G$ is a closed graph.

\begin{proposition}\label{relat-init}
Let $G$ be a connected closed graph on $n$ vertices and let $H_G$ be its initial graph. Then 
$${\rm v} (J_G)<\theta(G)\leq  {\rm v} (\init (J_G))={\rm v} (I(H_G)).$$
\end{proposition}

\begin{proof}
Let $G\neq\mathcal{K}_{n}$ be a closed graph. By Theorem \ref{v-number-clutters-graphs} there exists $A \in \mathcal{A}_{H_G}$ such that ${\rm v}(I(H_G))=|A|$. Let $A=\{v_1, \ldots,v_s\}$ and define $B=\{w_1, \ldots, w_s\}$ in the following way:\par
\begin{itemize}
\item If $v_i= x_{a_i}$ for some $a_i$, then $w_i:=y_{b_i}$ where $b_i=\max \{j \in [n] \mid \{a_i,j\} \in E(G)\}$.
\item If $v_i= y_{b_i}$ for some $b_i$, then $w_i:=x_{a_i}$ where $a_i=\min \{j \in [n] \mid \{j,b_i\} \in E(G)\}$.
\end{itemize}
Since $G$ is closed, each edge $\{v_i,w_i\}$ of $H_G$ defines a clique $C_i$ in $G$ on vertices $[a_i,b_i]:=\{a_i, a_{i}+1, \ldots, b_i\}$. We claim that $\mathcal{C}=\{C_1, \ldots, C_s\}$ is a clique cover of $G$, i.e., $V(G)= \bigcup\limits_{i=1}^{s} V(C_i)$. Let $k \in V(G)$. Since $G$ is connected, there exists $l \in V(G)$ such that $\{k,l\} \in E(G)$. Assume that $k<l$, so that $\{x_k, y_l\} \in E(H_G)$ (likewise, one can prove the case $k>l$). Since $N_{H_G}(A)$ is a minimal vertex cover of $H_G$, there exists $i \in \{1,\ldots,s\}$ such that either $x_k \in N_{H_G}(v_i)$ or $y_l \in N_{H_G}(v_i)$.\par
 If $x_k \in N_{H_G}(v_i)$, then $v_i=y_{b_i}$ and $\{k,b_i\} \in E(G)$. Therefore, by construction, $a_i \leq k$ and we have $k \in V(C_i)=[a_i,b_i]$. \par
 If $y_l \in N_{H_G}(v_i)$, then $v_i=x_{a_i}$ and $\{a_i,l\} \in E(G)$. Moreover, by the way $a_i$ and $b_i$ are defined, we get $a_i<l\leq b_i$. Thus, we have two possibilities for $k$. If $a_i \leq k <l \leq b_i$, then $k \in V(C_i)$ and we are done. If instead $k < a_i <l \leq b_i$, since $G$ is closed, we have that $\{k,a_i\} \in E(G)$ and $a_i<l$. Iterating the above procedure, we eventually find a clique $C_i$ such that $k\in V(C_i)$. This proves the claim and implies that $\theta(G)\leq {\rm v}(I(H))$.\par
 Now let $\mathcal{C}=\{C_1, \ldots, C_s\}$ be a clique cover of $G$, with $V(C_i):=[a_i,b_i]$.
Notice that, since $G$ is closed,  $a_i<a_{i+1} \leq b_i <b_{i+1}$ for every $i$. Let $e=\{r,t\} \notin E(G)$. Since $\mathcal{C}$ is a clique cover, we can always find cliques, say $C_{i_r}$ and $C_{i_t}$, such that $ r \in V(C_{i_r})$ and $ t \in V(C_{i_t})$. Moreover $i_r \neq i_t$ because $e=\{r,t\} \notin E(G)$ and the following is a path from $r$ to $t$
 $$ \{r, b_{i_r}\}, \{b_{i_r}, b_{i_r+1}\}, \ldots, \{b_{i_t-1},t\}.$$ 
This shows that, given $e=\{r,t\} \notin E(G)$, we can always find a path $P_{r,t}$ from $r$ to $t$ with vertices in $B=\{b_1,b_2, \ldots, b_{s-1}\}$. Thus, $B \in \mathcal{D}_{c}(G)$ and by Theorem \ref{comb_at_comple} we get 
$${\rm v} (J_G) \leq {\rm v}_{J_{\mathcal{K}_n}} (J_G)\leq s-1 = |\mathcal{C}|-1.$$
for every clique cover $\mathcal{C}$. In particular
$${\rm v} (J_G) \leq {\rm v}_{J_{\mathcal{K}_n}} (J_G)\leq \theta(G) -1.$$
 
  
 

If $G$ is the complete graph, then the set of neighbours of $x_{1}$ is a minimal vertex cover of $H_G$. Therefore $${\rm v} (J_G)=0<1= \theta(G)={\rm v} (I(H))={\rm v}(\init (J_G))$$
and we are done. 
\end{proof}
Given a graph $G$, the relation between $\IM (G)$ and ${\rm v}(I(G))$ has been already investigated. Grisalde, Reyes, and Villarreal studied the case of well-covered graphs {\rm \cite{Gri-Re-Vill}}. More recentely, Saha and Sengupta proved that if $G$ is a  bipartite graph then ${\rm v}(I(G))\leq\IM (G)$ {\rm \cite[Theorem 4.5]{Saha-Seng}. Below we include a proof of the latter result for the sake of completeness. This proof slightly differs from the one by the previous authors.

\begin{proposition}\label{init-Induced}
Let $G$ be a connected closed graph and let $H_G$ be its initial graph. Then
$$  {\rm v}(\init (J_G))={\rm v}(I(H_G)) \leq \IM (H_G)=\ell_G.$$
\end{proposition}

\begin{proof}
Let $\ell:=\ell_G$ and let $\{i_0,i_1\},\{i_1, i_2\}, \ldots,\{i_{\ell-1},i_{\ell}\}$ be a longest induced path in $G$. Define $$e_1:=\{x_{i_0},y_{i_1}\}, e_2:=\{x_{i_1},y_{i_2}\},\ldots, e_{\ell}:=\{x_{i_{\ell-1}},y_{i_{\ell}}\}$$
to be its corresponding induced matching in $H_G$. Consider the set $A=\{x_{i_0}, \ldots, x_{i_{\ell-1}}\}$. If we prove that $N_{H_G}(A)$ is a vertex cover, it follows from \cite[Lemma 3.4]{DR} that it is indeed minimal. Then, using \cite[Theorem 3.5]{DR}, we get
$$ {\rm v}(\init (J_G))={\rm v}(I(H_G)) \leq |A|= \IM (H_G)=\ell$$
and we are done.\par
For this purpose, let $e=\{x_r,y_s\} \in E(H_G)$. We want to prove that $ y_s \in N_{H_G}(A)$. If $x_r \in A$, there is nothing to prove. Assume that $x_r \notin A$.  If $ y_s \notin N_{H_G}(A)$, then there exists $k$ such that $\{x_r, y_{i_{k}}\} \in E(H_G)$ because $\ell$ is the induced matching number. Furthermore, $r> i_{k-1}$, otherwise we would have that $\{x_{i_{k-1}},y_{i_{k}}\} \in E(H_G)$, which contradicts the fact that $e_1, \ldots, e_{\ell}$ is an induced matching. Thus we have two possibilities.\par 
If $r<i_k<s$, since $G$ is closed, we get $\{x_{i_k},y_s\} \in E(H_G)$. Thus $ y_s \in N_{H_G}(A)$, a contradiction.\par
If $i_{k-1}<r<s< i_k$, since $G$ is closed, we get $\{x_{i_{k-1}},y_s\} \in E(H_G)$. Thus $ y_s \in N_{H_G}(A)$, a contradiction.
\end{proof}
Collecting together previous propositions, we obtain the following chain of inequalities. 
\begin{theorem}\label{relation_initial}
Let $G\neq\mathcal{K}_{n}$ be a connected closed graph and let $H_G$ be its initial graph. Then
$$ 
{\rm v} (J_G)<\theta(G)\leq  {\rm v} (\init (J_G))={\rm v}(I(H)) \leq \IM (H)=\ell_G = \reg (S/J_G)=\reg (S/\init(J_G)).
$$
If $G= \mathcal{K}_n$, then  $0={\rm v} (J_G)<{\rm v}(\init (J_G))=\reg (S/J_G)=\reg (S/\init(J_G))=1.$

\end{theorem}

\begin{remark}
It is worth noting that the results presented above assume that $G$ is closed with respect to the given labeling. If we change the labeling, then $ \init (J_G)$ may not represent the edge ideal of a bipartite graph in general.  However, $ \init (J_G)$ is squarefree {\rm \cite{HHHKR}}, thereby representing the edge ideal of a clutter, and {\rm \cite[Theorem 3.5]{DR}} can still be used to compare ${\rm v} (J_G)$ and ${\rm v}(\init (J_G))$.
\end{remark}

Except for the case of binomial edge ideals of closed graphs, not much is known about how v-numbers behave with respect to Gr\"obner degeneration. In this spirit one may ask

\begin{question}\label{question} Does the inequality ${\rm v}(I) \leq {\rm v}(\init (I))$ hold in a more general setting? 
\end{question}

This question arises naturally since v-numbers of monomial ideals have been already widely studied, and algebraic invariants of monomial ideals are often easier to compute.\par
Below is an example of a graph which is not closed but still its binomial edge ideal satisfies ${\rm v}(J_G) \leq {\rm v}(\init (J_G))\leq \reg(S/J_G)$ (see Section \ref{5} for more examples).
\begin{example}\label{ExampleNC}
Let $G$ be the non-closed graph in Figure \ref{correspondance}. Clearly, $B=\{ 2,3,4,5,7,9\}$ is a  minimum connected dominating set. Hence, ${\rm v}_{J_{\mathcal{K}_{10}}}(J_G)=\gamma_c(G)=6$. Also, $$\{\{1,2\},\{3,4\},\{5,6\},\{7,8\},\{9,10\}\}$$ is a minimum clique cover, so $\theta(G)=\ell=5$. A \emph{Macaulay2} computation (using Code \ref{v-number-procedure}) shows that ${\rm v} (J_G)=3$ (so, the minimum among all the localization of ${\rm v} (J_G)$ is not reached at the complete graph) and  ${\rm v} (\init(J_G))=6$. Moreover, since $G$ is a tree, it is itself a spanning tree and, by Corollary \ref{corollaryMain}, we have
$$ 
3={\rm v} (J_G)< 6= \gamma_c(G)={\rm v}_{J_{\mathcal{K}_{10}}}(J_G)={\rm v} (\init(J_G))\leq  \reg (S/J_G)=7.
$$
Notice that since $G$ is not closed, Theorem \ref{relation_initial} does not apply. In fact, the length of the longest induced path is $5$, so $\ell < {\rm v}_{J_{\mathcal{K}_{10}}}(J_G)$. 
\end{example}
\begin{figure}[H]
\begin{tikzpicture}[line width=.5pt,scale=0.75]
		\tikzstyle{every node}=[inner sep=1pt, minimum width=5.5pt] 
\tiny{
\node (1) at (-2,2){$\bullet$};
\node (2) at (-1.8,1) {$\bullet$};
\node (3) at (-1,.5) {$\bullet$};
\node (4) at (0,0) {$\bullet$};
\node (5) at (1.5,.5){$\bullet$};
\node (6) at (2,2) {$\bullet$};
\node (7) at (0,1) {$\bullet$};
\node (8) at (.5,2) {$\bullet$};
\node (9) at (0,-1.5){$\bullet$};
\node (10) at (-0.5,-2.5) {$\bullet$};
\node at (-2,2.3){$1$};
\node at (-1.6,1.1) {$2$};
\node at  (-1,.2){$3$};
\node at (0.2,-.1){$4$};
\node at (1.3,.7){$5$};
\node at (2.2,1.8) {$6$};
\node at  (.2,1){$7$};
\node at (.7,2){$8$};
\node at (-.2,-1.5){$9$};
\node at (-.8,-2.5) {$10$};
\draw[-,line width=1pt] (1) to (2);
\draw[-,line width=1pt] (2) to (3);
\draw[-,line width=1pt] (3) -- (4);
\draw[-,line width=1pt] (4) -- (5);
\draw[-,line width=1pt] (5) -- (6);
\draw[-,line width=1pt] (4) -- (5);
\draw[-,line width=1pt] (5) -- (6);
\draw[-,line width=1pt] (7) -- (8);
\draw[-,line width=1pt] (4) -- (7);
\draw[-,line width=1pt] (4) -- (9);
\draw[-,line width=1pt] (9) -- (10);
}
\end{tikzpicture}
\caption{}\label{correspondance}
\end{figure}

In the case of trees, by Corollary \ref{corollaryMain} and {\rm \cite[Theorem 4.1]{Jay-Nara}} the localization of the v-number at the complete graph is a better bound for the regularity compared to $\ell$. One may ask if this is the case for other families of graphs.

\begin{question}\label{question2} For which families of graphs is ${\rm v}_{J_{\mathcal{K}_{n}}}(J_G)$ a better lower bound for the regularity than $\ell$? 
\end{question}

\section{Some examples} \label{5}
We now discuss some examples related to our previous results. For this purpose, we also use Code \ref{v-number-procedure} to compute v-numbers in \emph{Macaulay2}. \par 
Throughout this section $\prec$ is the lexicographic order given by $x_{1}\succ\dots\succ x_{n}\succ y_{1}\succ\dots\succ y_{n}$.

\subsection{Complete multipartite graphs}
We start by computing the v-number of a complete multipartite graph $\mathcal{K}_{a_1,\ldots,a_r}$.

\begin{lemma}{\rm \cite[Lemma 2.2]{OhtBin}}\label{pri-bip-com}
Let $G=\mathcal{K}_{a_1,\ldots,a_r}$ with $a_1\leq \ldots \leq a_r$ be a complete $r$-partite graph on vertex set $V=\bigsqcup\limits_{l=1}^{r} V_l=[n]$. Set $s:= |\{l \mid a_l=1\}|$. Then, the minimal primary decomposition of $J_G$ is given by
$$J_{G}=J_{\mathcal{K}_{n}}\cap P_{s+1} \cap \ldots \cap P_{r}$$
where $P_{l}:=(x_{v},y_{v} \mid v \notin V_l)$ for $l=s+1, \ldots, r$. 
\end{lemma}


\begin{proposition}
Let $G=\mathcal{K}_{a_1,\ldots,a_r}$ with $a_1\leq \ldots \leq a_r$ be a complete $r$-partite graph on vertex set $V=\bigsqcup\limits_{l=1}^{r} V_l=[n]$. Then,
$$
{\rm v}(J_{\mathcal{K}_{m,n}})=\begin{cases}
1, & \hbox{ if } a_1=1, \\
2={\rm reg}(S/J_{G}), & \hbox{ otherwise}.
\end{cases}
$$

\end{proposition}

\begin{proof}
Assume that $a_1=1$. By Theorem \ref{comb_at_comple}, we have that ${\rm v}_{J_{\mathcal{K}_{n+1}}}(J_{\mathcal{K}_{a_1,\ldots,a_r}})= \gamma_c(\mathcal{K}_{a_1,\ldots,a_r})=1$. Thus, since $J_{\mathcal{K}_{a_1,\ldots,a_r}}$ is not prime, ${\rm v}(J_{\mathcal{K}_{a_1,\ldots,a_r}})=1$. Now, let $a_1>1$. By Lemma \ref{pri-bip-com} the minimal primes of  $J_{\mathcal{K}_{a_1,\ldots,a_r}}$ are $J_{\mathcal{K}_{n}}$ and $P_{l}$ for $l=s+1, \ldots, r$. Since the generators of $J_{\mathcal{K}_{a_1,\ldots,a_r}}$ are quadratic binomials of the form $x_iy_j-x_jy_i$, then there cannot be an element $f\in S_{1}$ such that $(J_{\mathcal{K}_{a_1,\ldots,a_r}}:f)=P_{l}$. Therefore, by definition of v-number, we have that 
$$
2\leq {\rm v}_{P_{l}}(J_{\mathcal{K}_{a_1,\ldots,a_r}}),\quad \forall l \in \{s+1, \ldots, r\}.
$$
Moreover, it is easy to see that every two vertices $\{v,w\}$ that form an edge (i.e., they belong to different $V_i$'s) are a connected dominating set. Since by our assumption on $a_1$ there is no vertex of degree $n-1$, it follows that $\gamma_c(\mathcal{K}_{a_1,\ldots,a_r})=2$. So, by Theorem \ref{comb_at_comple} $${\rm v}_{J_{\mathcal{K}_{a_1,\ldots,a_r}}}(J_{\mathcal{K}_{a_1,\ldots,a_r}})=\gamma_c(\mathcal{K}_{a_1,\ldots,a_r})= 2.$$  Hence, 
$$
{\rm v}(J_{\mathcal{K}_{a_1,\ldots,a_r}})=2={\rm reg}(S/J_{\mathcal{K}_{a_1,\ldots,a_r}})
$$
 where the last equality follows from {\rm \cite[Theorem 1.1]{Sche-Zafar}} and {\rm \cite[Proposition 3.4]{Kiani-Mada2}}.
\end{proof}

\subsection{Graphs with $5$ vertices}
Driven by Question \ref{question}, we compute the vector $$({\rm v}(J_{G}),\: {\rm v}({\rm in}_{\prec}(J_{G})),\:{\rm reg}(S/J_{G}))$$ for all connected graphs with $5$ vertices. In all these cases, we obtain ${\rm v}(J_{G})\leq {\rm v}({\rm in}_{\prec}(J_{G}))\leq {\rm reg}(S/J_{G})$.

\begin{table}[H]
\begin{tabular}{|c|c|c|c|c|c|}
\cline{3-3}  
\multicolumn{1}{c}{}& &
\begin{tikzpicture}[line width=.5pt,scale=0.75]
\tikzstyle{every node}=[inner sep=1pt, minimum width=5.5pt] 
\tiny{
\node (1) at (0,1){$\bullet$};
\node (2) at (1,0) {$\bullet$};
\node (3) at (.5,-1) {$\bullet$};
\node (4) at (-.5,-1){$\bullet$};
\node (5) at (-1,0) {$\bullet$};
\node at (0,-1.7){${(0,1,1)}$};
\node at (0,1.2){$1$};
\node at (1.2,0) {$2$};
\node at  (.5,-1.2){$3$};
\node at (-.5,-1.2){$4$};
\node at (-1.2,0){$5$};
\draw[-,line width=1pt] (1) to (2);
\draw[-,line width=1pt] (1) to (3);
\draw[-,line width=1pt] (1) -- (4);
\draw[-,line width=1pt] (1) -- (5);
\draw[-,line width=1pt] (2) -- (3);
\draw[-,line width=1pt] (2) -- (4);
\draw[-,line width=1pt] (2) -- (5);
\draw[-,line width=1pt] (3) -- (4);
\draw[-,line width=1pt] (3) -- (5);
\draw[-,line width=1pt] (4) -- (5);
}
\end{tikzpicture}\\
\hline \begin{tikzpicture}[line width=.5pt,scale=0.75]
\tikzstyle{every node}=[inner sep=1pt, minimum width=5.5pt] 
\tiny{
\node (1) at (0,1){$\bullet$};
\node (2) at (1,0) {$\bullet$};
\node (3) at (.5,-1) {$\bullet$};
\node (4) at (-.5,-1){$\bullet$};
\node (5) at (-1,0) {$\bullet$};
\node at (0,-1.7){${(3,3,3)}$};
\node at (0,1.2){$1$};
\node at (1.2,0) {$2$};
\node at  (.5,-1.2){$3$};
\node at (-.5,-1.2){$4$};
\node at (-1.2,0){$5$};
\draw[-,line width=1pt] (1) to (2);
\draw[-,line width=1pt] (1) to (5);
\draw[-,line width=1pt] (2) -- (3);
\draw[-,line width=1pt] (3) -- (4);
\draw[-,line width=1pt] (4) -- (5);
}
\end{tikzpicture}&\begin{tikzpicture}[line width=.5pt,scale=0.75]
\tikzstyle{every node}=[inner sep=1pt, minimum width=5.5pt] 
\tiny{
\node (1) at (0,1){$\bullet$};
\node (2) at (1,0) {$\bullet$};
\node (3) at (.5,-1) {$\bullet$};
\node (4) at (-.5,-1){$\bullet$};
\node (5) at (-1,0) {$\bullet$};
\node at (0,-1.7){${(1,1,2)}$};
\node at (0,1.2){$1$};
\node at (1.2,0) {$2$};
\node at  (.5,-1.2){$3$};
\node at (-.5,-1.2){$4$};
\node at (-1.2,0){$5$};
\draw[-,line width=1pt] (1) to (2);
\draw[-,line width=1pt] (1) to (3);
\draw[-,line width=1pt] (1) -- (4);
\draw[-,line width=1pt] (1) -- (5);
\draw[-,line width=1pt] (2) -- (3);
\draw[-,line width=1pt] (2) -- (5);
\draw[-,line width=1pt] (3) -- (4);
\draw[-,line width=1pt] (4) -- (5);
}
\end{tikzpicture}&
\begin{tikzpicture}[line width=.5pt,scale=0.75]
\tikzstyle{every node}=[inner sep=1pt, minimum width=5.5pt] 
\tiny{
\node (1) at (0,1){$\bullet$};
\node (2) at (1,0) {$\bullet$};
\node (3) at (.5,-1) {$\bullet$};
\node (4) at (-.5,-1){$\bullet$};
\node (5) at (-1,0) {$\bullet$};
\node at (0,-1.7){${(2,2,2)}$};
\node at (0,1.2){$1$};
\node at (1.2,0) {$2$};
\node at  (.5,-1.2){$3$};
\node at (-.5,-1.2){$4$};
\node at (-1.2,0){$5$};
\draw[-,line width=1pt] (1) to (3);
\draw[-,line width=1pt] (1) -- (5);
\draw[-,line width=1pt] (2) -- (3);
\draw[-,line width=1pt] (2) -- (5);
\draw[-,line width=1pt] (3) -- (4);
\draw[-,line width=1pt] (4) -- (5);
}
\end{tikzpicture}&\begin{tikzpicture}[line width=.5pt,scale=0.75]
\tikzstyle{every node}=[inner sep=1pt, minimum width=5.5pt] 
\tiny{
\node (1) at (0,1){$\bullet$};
\node (2) at (1,0) {$\bullet$};
\node (3) at (.5,-1) {$\bullet$};
\node (4) at (-.5,-1){$\bullet$};
\node (5) at (-1,0) {$\bullet$};
\node at (0,-1.7){${(1,1,2)}$};
\node at (0,1.2){$1$};
\node at (1.2,0) {$2$};
\node at  (.5,-1.2){$3$};
\node at (-.5,-1.2){$4$};
\node at (-1.2,0){$5$};
\draw[-,line width=1pt] (1) to (2);
\draw[-,line width=1pt] (1) -- (5);
\draw[-,line width=1pt] (2) -- (3);
\draw[-,line width=1pt] (2) -- (4);
\draw[-,line width=1pt] (2) -- (5);
\draw[-,line width=1pt] (3) -- (5);
\draw[-,line width=1pt] (4) -- (5);
}
\end{tikzpicture}&\begin{tikzpicture}[line width=.5pt,scale=0.75]
\tikzstyle{every node}=[inner sep=1pt, minimum width=5.5pt] 
\tiny{
\node (1) at (0,1){$\bullet$};
\node (2) at (1,0) {$\bullet$};
\node (3) at (.5,-1) {$\bullet$};
\node (4) at (-.5,-1){$\bullet$};
\node (5) at (-1,0) {$\bullet$};
\node at (0,-1.7){${(1,2,2)}$};
\node at (0,1.2){$2$};
\node at (1.2,0) {$1$};
\node at  (.5,-1.2){$3$};
\node at (-.5,-1.2){$4$};
\node at (-1.2,0){$5$};
\draw[-,line width=1pt] (1) to (2);
\draw[-,line width=1pt] (1) to (3);
\draw[-,line width=1pt] (1) -- (4);
\draw[-,line width=1pt] (1) -- (5);
\draw[-,line width=1pt] (2) -- (3);
\draw[-,line width=1pt] (2) -- (4);
\draw[-,line width=1pt] (3) -- (4);
\draw[-,line width=1pt] (3) -- (5);
\draw[-,line width=1pt] (4) -- (5);
}
\end{tikzpicture}\\
\hline 
\begin{tikzpicture}[line width=.5pt,scale=0.75]
\tikzstyle{every node}=[inner sep=1pt, minimum width=5.5pt] 
\tiny{
\node (1) at (0,1){$\bullet$};
\node (2) at (1,0) {$\bullet$};
\node (3) at (.5,-1) {$\bullet$};
\node (4) at (-.5,-1){$\bullet$};
\node (5) at (-1,0) {$\bullet$};
\node at (0,-1.7){${(1,1,2)}$};
\node at (0,1.2){$1$};
\node at (1.2,0) {$2$};
\node at  (.5,-1.2){$3$};
\node at (-.5,-1.2){$4$};
\node at (-1.2,0){$5$};
\draw[-,line width=1pt] (1) to (2);
\draw[-,line width=1pt] (1) to (3);
\draw[-,line width=1pt] (1) -- (4);
\draw[-,line width=1pt] (1) -- (5);
}
\end{tikzpicture}&\begin{tikzpicture}[line width=.5pt,scale=0.75]
\tikzstyle{every node}=[inner sep=1pt, minimum width=5.5pt] 
\tiny{
\node (1) at (0,1){$\bullet$};
\node (2) at (1,0) {$\bullet$};
\node (3) at (.5,-1) {$\bullet$};
\node (4) at (-.5,-1){$\bullet$};
\node (5) at (-1,0) {$\bullet$};
\node at (0,-1.7){${(1,2,2)}$};
\node at (0,1.2){$1$};
\node at (1.2,0) {$5$};
\node at  (.5,-1.2){$4$};
\node at (-.5,-1.2){$3$};
\node at (-1.2,0){$2$};
\draw[-,line width=1pt] (1) -- (5);
\draw[-,line width=1pt] (2) -- (3);
\draw[-,line width=1pt] (2) -- (4);
\draw[-,line width=1pt] (2) -- (5);
\draw[-,line width=1pt] (3) -- (4);
\draw[-,line width=1pt] (3) -- (5);
\draw[-,line width=1pt] (4) -- (5);
}
\end{tikzpicture}&\begin{tikzpicture}[line width=.5pt,scale=0.75]
\tikzstyle{every node}=[inner sep=1pt, minimum width=5.5pt] 
\tiny{
\node (1) at (0,1){$\bullet$};
\node (2) at (1,0) {$\bullet$};
\node (3) at (.5,-1) {$\bullet$};
\node (4) at (-.5,-1){$\bullet$};
\node (5) at (-1,0) {$\bullet$};
\node at (0,-1.7){${(1,2,2)}$};
\node at (0,1.2){$1$};
\node at (1.2,0) {$3$};
\node at  (.5,-1.2){$4$};
\node at (-.5,-1.2){$5$};
\node at (-1.2,0){$2$};
\draw[-,line width=1pt] (1) to (2);
\draw[-,line width=1pt] (1) -- (5);
\draw[-,line width=1pt] (2) -- (3);
\draw[-,line width=1pt] (2) -- (4);
\draw[-,line width=1pt] (2) -- (5);
\draw[-,line width=1pt] (3) -- (4);
\draw[-,line width=1pt] (3) -- (5);
\draw[-,line width=1pt] (4) -- (5);
}
\end{tikzpicture}&
\begin{tikzpicture}[line width=.5pt,scale=0.75]
\tikzstyle{every node}=[inner sep=1pt, minimum width=5.5pt] 
\tiny{
\node (1) at (0,1){$\bullet$};
\node (2) at (1,0) {$\bullet$};
\node (3) at (.5,-1) {$\bullet$};
\node (4) at (-.5,-1){$\bullet$};
\node (5) at (-1,0) {$\bullet$};
\node at (0,-1.7){${(2,2,3)}$};
\node at (0,1.2){$1$};
\node at (1.2,0) {$2$};
\node at  (.5,-1.2){$3$};
\node at (-.5,-1.2){$4$};
\node at (-1.2,0){$5$};
\draw[-,line width=1pt] (1) to (2);
\draw[-,line width=1pt] (1) -- (5);
\draw[-,line width=1pt] (1) -- (3);
\draw[-,line width=1pt] (4) -- (5);
}
\end{tikzpicture}&\begin{tikzpicture}[line width=.5pt,scale=0.75]
\tikzstyle{every node}=[inner sep=1pt, minimum width=5.5pt] 
\tiny{
\node (1) at (0,1){$\bullet$};
\node (2) at (1,0) {$\bullet$};
\node (3) at (.5,-1) {$\bullet$};
\node (4) at (-.5,-1){$\bullet$};
\node (5) at (-1,0) {$\bullet$};
\node at (0,-1.7){${(1,2,2)}$};
\node at (0,1.2){$2$};
\node at (1.2,0) {$1$};
\node at  (.5,-1.2){$5$};
\node at (-.5,-1.2){$4$};
\node at (-1.2,0){$3$};
\draw[-,line width=1pt] (1) to (2);
\draw[-,line width=1pt] (1) -- (5);
\draw[-,line width=1pt] (2) -- (5);
\draw[-,line width=1pt] (3) -- (4);
\draw[-,line width=1pt] (3) -- (5);
\draw[-,line width=1pt] (4) -- (5);
}
\end{tikzpicture}\\
\hline \begin{tikzpicture}[line width=.5pt,scale=0.75]
\tikzstyle{every node}=[inner sep=1pt, minimum width=5.5pt] 
\tiny{
\node (1) at (0,1){$\bullet$};
\node (2) at (1,0) {$\bullet$};
\node (3) at (.5,-1) {$\bullet$};
\node (4) at (-.5,-1){$\bullet$};
\node (5) at (-1,0) {$\bullet$};
\node at (0,-1.7){${(1,2,3)}$};
\node at (0,1.2){$1$};
\node at (1.2,0) {$3$};
\node at  (.5,-1.2){$5$};
\node at (-.5,-1.2){$4$};
\node at (-1.2,0){$2$};
\draw[-,line width=1pt] (1) to (2);
\draw[-,line width=1pt] (1) -- (5);
\draw[-,line width=1pt] (2) -- (3);
\draw[-,line width=1pt] (2) -- (4);
\draw[-,line width=1pt] (2) -- (5);
\draw[-,line width=1pt] (3) -- (4);
\draw[-,line width=1pt] (4) -- (5);
}
\end{tikzpicture}&
\begin{tikzpicture}[line width=.5pt,scale=0.75]
\tikzstyle{every node}=[inner sep=1pt, minimum width=5.5pt] 
\tiny{
\node (1) at (0,1){$\bullet$};
\node (2) at (1,0) {$\bullet$};
\node (3) at (.5,-1) {$\bullet$};
\node (4) at (-.5,-1){$\bullet$};
\node (5) at (-1,0) {$\bullet$};
\node at (0,-1.7){${(1,1,2)}$};
\node at (0,1.2){$1$};
\node at (1.2,0) {$2$};
\node at  (.5,-1.2){$3$};
\node at (-.5,-1.2){$4$};
\node at (-1.2,0){$5$};
\draw[-,line width=1pt] (1) -- (5);
\draw[-,line width=1pt] (2) -- (5);
\draw[-,line width=1pt] (3) -- (4);
\draw[-,line width=1pt] (3) -- (5);
\draw[-,line width=1pt] (4) -- (5);
}
\end{tikzpicture}&\begin{tikzpicture}[line width=.5pt,scale=0.75]
\tikzstyle{every node}=[inner sep=1pt, minimum width=5.5pt] 
\tiny{
\node (1) at (0,1){$\bullet$};
\node (2) at (1,0) {$\bullet$};
\node (3) at (.5,-1) {$\bullet$};
\node (4) at (-.5,-1){$\bullet$};
\node (5) at (-1,0) {$\bullet$};
\node at (0,-1.7){${(2,3,3)}$};
\node at (0,1.2){$2$};
\node at (1.2,0) {$1$};
\node at  (.5,-1.2){$3$};
\node at (-.5,-1.2){$5$};
\node at (-1.2,0){$4$};
\draw[-,line width=1pt] (1) -- (5);
\draw[-,line width=1pt] (2) -- (3);
\draw[-,line width=1pt] (3) -- (4);
\draw[-,line width=1pt] (3) -- (5);
\draw[-,line width=1pt] (4) -- (5);
}
\end{tikzpicture}&\begin{tikzpicture}[line width=.5pt,scale=0.75]
\tikzstyle{every node}=[inner sep=1pt, minimum width=5.5pt] 
\tiny{
\node (1) at (0,1){$\bullet$};
\node (2) at (1,0) {$\bullet$};
\node (3) at (.5,-1) {$\bullet$};
\node (4) at (-.5,-1){$\bullet$};
\node (5) at (-1,0) {$\bullet$};
\node at (0,-1.7){${(2,3,3)}$};
\node at (0,1.2){$5$};
\node at (1.2,0) {$3$};
\node at  (.5,-1.2){$1$};
\node at (-.5,-1.2){$2$};
\node at (-1.2,0){$4$};
\draw[-,line width=1pt] (1) -- (5);
\draw[-,line width=1pt] (2) -- (3);
\draw[-,line width=1pt] (2) -- (4);
\draw[-,line width=1pt] (2) -- (5);
\draw[-,line width=1pt] (3) -- (4);
\draw[-,line width=1pt] (4) -- (5);
}
\end{tikzpicture}&
\begin{tikzpicture}[line width=.5pt,scale=0.75]
\tikzstyle{every node}=[inner sep=1pt, minimum width=5.5pt] 
\tiny{
\node (1) at (0,1){$\bullet$};
\node (2) at (1,0) {$\bullet$};
\node (3) at (.5,-1) {$\bullet$};
\node (4) at (-.5,-1){$\bullet$};
\node (5) at (-1,0) {$\bullet$};
\node at (0,-1.7){${(1,1,2)}$};
\node at (0,1.2){$1$};
\node at (1.2,0) {$2$};
\node at  (.5,-1.2){$3$};
\node at (-.5,-1.2){$4$};
\node at (-1.2,0){$5$};
\draw[-,line width=1pt] (1) -- (5);
\draw[-,line width=1pt] (2) -- (3);
\draw[-,line width=1pt] (2) -- (5);
\draw[-,line width=1pt] (3) -- (4);
\draw[-,line width=1pt] (3) -- (5);
\draw[-,line width=1pt] (4) -- (5);
}
\end{tikzpicture}\\
\hline \begin{tikzpicture}[line width=.5pt,scale=0.75]
\tikzstyle{every node}=[inner sep=1pt, minimum width=5.5pt] 
\tiny{
\node (1) at (0,1){$\bullet$};
\node (2) at (1,0) {$\bullet$};
\node (3) at (.5,-1) {$\bullet$};
\node (4) at (-.5,-1){$\bullet$};
\node (5) at (-1,0) {$\bullet$};
\node at (0,-1.7){${(2,4,4)}$};
\node at (0,1.2){$3$};
\node at (1.2,0) {$2$};
\node at  (.5,-1.2){$1$};
\node at (-.5,-1.2){$5$};
\node at (-1.2,0){$4$};
\draw[-,line width=1pt] (1) to (2);
\draw[-,line width=1pt] (1) -- (5);
\draw[-,line width=1pt] (2) -- (3);
\draw[-,line width=1pt] (4) -- (5);
}
\end{tikzpicture}&\begin{tikzpicture}[line width=.5pt,scale=0.75]
\tikzstyle{every node}=[inner sep=1pt, minimum width=5.5pt] 
\tiny{
\node (1) at (0,1){$\bullet$};
\node (2) at (1,0) {$\bullet$};
\node (3) at (.5,-1) {$\bullet$};
\node (4) at (-.5,-1){$\bullet$};
\node (5) at (-1,0) {$\bullet$};
\node at (0,-1.7){${(2,3,3)}$};
\node at (0,1.2){$2$};
\node at (1.2,0) {$1$};
\node at  (.5,-1.2){$4$};
\node at (-.5,-1.2){$5$};
\node at (-1.2,0){$3$};
\draw[-,line width=1pt] (1) to (2);
\draw[-,line width=1pt] (1) -- (5);
\draw[-,line width=1pt] (3) -- (4);
\draw[-,line width=1pt] (3) -- (5);
\draw[-,line width=1pt] (4) -- (5);
}
\end{tikzpicture}
&\begin{tikzpicture}[line width=.5pt,scale=0.75]
\tikzstyle{every node}=[inner sep=1pt, minimum width=5.5pt] 
\tiny{
\node (1) at (0,1){$\bullet$};
\node (2) at (1,0) {$\bullet$};
\node (3) at (.5,-1) {$\bullet$};
\node (4) at (-.5,-1){$\bullet$};
\node (5) at (-1,0) {$\bullet$};
\node at (0,-1.7){${(2,2,2)}$};
\node at (0,1.2){$1$};
\node at (1.2,0) {$2$};
\node at  (.5,-1.2){$3$};
\node at (-.5,-1.2){$4$};
\node at (-1.2,0){$5$};
\draw[-,line width=1pt] (1) to (2);
\draw[-,line width=1pt] (1) to (3);
\draw[-,line width=1pt] (1) -- (5);
\draw[-,line width=1pt] (2) -- (3);
\draw[-,line width=1pt] (2) -- (5);
\draw[-,line width=1pt] (3) -- (4);
\draw[-,line width=1pt] (4) -- (5);
}
\end{tikzpicture}&\begin{tikzpicture}[line width=.5pt,scale=0.75]
\tikzstyle{every node}=[inner sep=1pt, minimum width=5.5pt] 
\tiny{
\node (1) at (0,1){$\bullet$};
\node (2) at (1,0) {$\bullet$};
\node (3) at (.5,-1) {$\bullet$};
\node (4) at (-.5,-1){$\bullet$};
\node (5) at (-1,0) {$\bullet$};
\node at (0,-1.7){${(2,2,3)}$};
\node at (0,1.2){$1$};
\node at (1.2,0) {$2$};
\node at  (.5,-1.2){$4$};
\node at (-.5,-1.2){$5$};
\node at (-1.2,0){$3$};
\draw[-,line width=1pt] (1) to (2);
\draw[-,line width=1pt] (1) -- (5);
\draw[-,line width=1pt] (2) -- (3);
\draw[-,line width=1pt] (2) -- (5);
\draw[-,line width=1pt] (3) -- (4);
\draw[-,line width=1pt] (4) -- (5);
}
\end{tikzpicture}&\begin{tikzpicture}[line width=.5pt,scale=0.75]
\tikzstyle{every node}=[inner sep=1pt, minimum width=5.5pt] 
\tiny{
\node (1) at (0,1){$\bullet$};
\node (2) at (1,0) {$\bullet$};
\node (3) at (.5,-1) {$\bullet$};
\node (4) at (-.5,-1){$\bullet$};
\node (5) at (-1,0) {$\bullet$};
\node at (0,-1.7){${(2,2,3)}$};
\node at (0,1.2){$1$};
\node at (1.2,0) {$2$};
\node at  (.5,-1.2){$5$};
\node at (-.5,-1.2){$4$};
\node at (-1.2,0){$3$};
\draw[-,line width=1pt] (1) -- (5);
\draw[-,line width=1pt] (2) -- (3);
\draw[-,line width=1pt] (2) -- (5);
\draw[-,line width=1pt] (3) -- (4);
\draw[-,line width=1pt] (4) -- (5);
}
\end{tikzpicture}\\
\hline
\end{tabular}
\caption{}
\label{five_vertices}
\end{table}


\subsection{Other examples}
We conclude with some examples of graphs with more than 5 vertices.

\begin{example}\label{cycle}

Let $J_{C_{6}}$ be the binomial edge ideal of the cycle with six vertices. 

\begin{figure}[H]
\centering
\begin{tikzpicture}[line width=.5pt,scale=0.75]
		\tikzstyle{every node}=[inner sep=1pt, minimum width=5.5pt] 
\tiny{
\node (1) at (-1,1){$\bullet$};
\node (6) at (1,1) {$\bullet$};
\node (2) at (-2,0) {$\bullet$};
\node (5) at (2,0){$\bullet$};
\node (3) at (-1,-1) {$\bullet$};
\node (4) at (1,-1) {$\bullet$};
\node at (-1,1.3){$1$};
\node at (1,1.3) {$6$};
\node at  (-2.3,0){$2$};
\node at (2.3,0){$5$};
\node at (-1,-1.3){$3$};
\node at (1,-1.3) {$4$};
\draw[-,line width=1pt] (1) to (2);
\draw[-,line width=1pt] (2) to (3);
\draw[-,line width=1pt] (3) -- (4);
\draw[-,line width=1pt] (4) -- (5);
\draw[-,line width=1pt] (5) -- (6);
\draw[-,line width=1pt] (6) -- (1);
}
\end{tikzpicture}
\caption{}\label{Example3}
\end{figure}

If $K=\mathbb{Q}$, using Theorem \ref{comb_at_comple}, and Code \ref{v-number-procedure} we obtain 
$$
{\rm v}_{J_{\mathcal{K}_{6}}}(J_{C_{6}})={\rm v}(J_{C_6})={\rm v}({\rm in}_{\prec}(J_{C_{6}}))=4=\reg(S/J_{C_6})=\reg(S/{\rm in}_{\prec}(J_{C_6})).
$$
Similarly, for $C_7$ we get
$$
{\rm v}_{J_{\mathcal{K}_{7}}}(J_{C_{7}})={\rm v}(J_{C_7})={\rm v}({\rm in}_{\prec}(J_{C_{7}}))=5=\reg(S/J_{C_7})=\reg(S/{\rm in}_{\prec}(J_{C_7})).
$$
These values, together with the ones obtained above for $C_5$, suggest that for any cycle $C_n$, one may have
$$
{\rm v}_{J_{\mathcal{K}_{n}}}(J_{C_{n}})={\rm v}(J_{C_{n}})={\rm v}({\rm in}_{\prec}(J_{C_{n}}))=n-2=\reg(S/J_{C_n})=\reg(S/{\rm in}_{\prec}(J_{C_n})).
$$
\end{example}

However, in general, both inequalities in ${\rm v}(J_{G})\leq {\rm v}({\rm in}_{\prec}(J_{G}))\leq {\rm reg}(S/J_{G})$ can be strict, as shown in the next two examples.
\begin{example}[Closed graph]\label{close_graph_ex}
Let $G$ be the following closed graph, 
\begin{figure}[H]
\centering
\begin{tikzpicture}[line width=.5pt,scale=0.75]
		\tikzstyle{every node}=[inner sep=1pt, minimum width=5.5pt] 
\tiny{
\node (4) at (-1,0){$\bullet$};
\node (5) at (1,0) {$\bullet$};
\node (3) at (-2,1) {$\bullet$};
\node (2) at (-2,-1){$\bullet$};
\node (1) at (-3,0) {$\bullet$};
\node (7) at (2,1) {$\bullet$};
\node (6) at (2,-1) {$\bullet$};
\node (8) at (3,0) {$\bullet$};
\node at (-1,0.3){$4$};
\node at (1,0.3) {$5$};
\node at  (-2,1.3){$3$};
\node at (-2,-1.3){$2$};
\node at (-3.3,0){$1$};
\node at (2,1.3) {$7$};
\node at (2,-1.3) {$6$};
\node at (3.3,0) {$8$};
\draw[-,line width=1pt] (1) to (2);
\draw[-,line width=1pt] (1) to (3);
\draw[-,line width=1pt] (3) -- (2);
\draw[-,line width=1pt] (4) -- (3);
\draw[-,line width=1pt] (4) -- (2);
\draw[-,line width=1pt] (4) -- (5);
\draw[-,line width=1pt] (5) -- (6);
\draw[-,line width=1pt] (5) -- (7);
\draw[-,line width=1pt] (7) -- (6);
\draw[-,line width=1pt] (7) -- (8);
\draw[-,line width=1pt] (8) -- (6);

}
\end{tikzpicture}
\caption{}\label{Example1}
\end{figure}
\end{example}

The binomial edge ideal of $G$ is 
\begin{align*}
J_{G}=(&x_1y_2-x_2y_1,x_1y_3-x_3y_1,x_2y_3-x_3y_2,x_2y_4-x_4y_2,x_3y_4-x_4y_3,x_4y_5-x_5y_4,x_5y_6-x_6y_5,\\
&x_5y_7-x_7y_5,x_6y_7-x_7y_6,x_6y_8-x_8y_6,x_7y_8-x_8y_7)\subseteq S=K[x_{1},\dots,x_{8},y_{1},\dots,y_{8}].
\end{align*}

Proposition \ref{regBEI}, Theorem \ref{comb_at_comple}, and Theorem \ref{relation_initial} imply that 
\begin{align*}
{\rm v}_{J_{\mathcal{K}_{8}}}(J_{G})&=4\\
{\rm v}({\rm in}_{\prec}(J_{G}))=\theta(G)&=5=\reg(S/J_{G})=\reg(S/{\rm in}_{\prec}(J_{G}))=\ell_G .
\end{align*}

Moreover, using Code \ref{v-number-procedure}, we obtain ${\rm v}(J_{G})=3$. So, in this case, all the aforementioned inequalities are strict.


\begin{example}[Non-closed graph]\label{non-clos-six}
Consider the following non-closed graph $G$ with six vertices,

\begin{figure}[H]
\centering
\begin{tikzpicture}[line width=.5pt,scale=0.75]
		\tikzstyle{every node}=[inner sep=1pt, minimum width=5.5pt] 
\tiny{
\node (1) at (0,2){$\bullet$};
\node (2) at (-1,1) {$\bullet$};
\node (3) at (-1,-1) {$\bullet$};
\node (4) at (1,-1){$\bullet$};
\node (5) at (1,1) {$\bullet$};
\node (6) at (0,0) {$\bullet$};
\node at (0,2.3){$1$};
\node at (-1,1.3) {$2$};
\node at  (-1,-1.3){$3$};
\node at (1,-1.3){$4$};
\node at (1,1.3){$5$};
\node at (0,0.3) {$6$};
\draw[-,line width=1pt] (1) to (2);
\draw[-,line width=1pt] (2) to (3);
\draw[-,line width=1pt] (3) -- (4);
\draw[-,line width=1pt] (4) -- (5);
\draw[-,line width=1pt] (5) -- (1);
\draw[-,line width=1pt] (6) -- (5);
\draw[-,line width=1pt] (6) -- (2);

}
\end{tikzpicture}
\caption{}\label{Example2}
\end{figure}

The binomial edge ideal of $G$ is given by 
\begin{align*}
J_{G}&=(x_{1}y_{2}-x_{2}y_{1},x_{2}y_{3}-x_{3}y_{2},x_{3}y_{4}-x_{4}y_{3},x_{4}y_{5}-x_{5}y_{4},x_{1}y_{5}-x_{5}y_{1},x_{2}y_{6}-x_{6}y_{2},x_{5}y_{6}-x_{6}y_{5})\\
&\subseteq S=K[x_{1},\dots,x_{6},y_{1},\dots,y_{6}].
\end{align*}
If $K=\mathbb{Q}$, using Theorem \ref{comb_at_comple} and Code \ref{v-number-procedure}, we obtain 
\begin{align*}
{\rm v}(J_{G})=2&<3={\rm v}({\rm in}_{\prec}(J_{G}))={\rm v}_{J_{\mathcal{K}_{6}}}(J_{G}),\\
\reg(S/J_{G})&=3=\reg(S/{\rm in}_{\prec}(J_{G})).
\end{align*}

\end{example}


\begin{procedure}\label{v-number-procedure}

The following function computes the v-number of any graded ideal. We have used this code to compute the v-number of the binomial edge ideals presented in this section.

\begin{verbatim}
--================= v-number function ========================
vnumber = method(TypicalValue => ZZ);
vnumber (Ideal) := (I) -> ( 
  L := ass I;
  G := apply(0..#L-1,n->gens gb ideal(flatten mingens(quotient(I,L#n)/I)));
  N := apply(G,i->flatten entries i);
  F := apply(0..#L-1,i->apply(N#i,x-> if not quotient(I,x)==L#i then 0 else x)
    -set{0});
    min flatten degrees ideal(flatten F)
    )
\end{verbatim}
\end{procedure}

\begin{procedure}\label{gammafunction}
The next code computes the connected domination number of a graph $G$ using the algebraic description found in Theorem \ref{comb_at_comple}.

\begin{verbatim}
loadPackage "Graphs"
loadPackage "BinomialEdgeIdeals". 
--================= Gammac function ========================
Gammac = method(TypicalValue => ZZ);
Gammac (Graph) := (G) -> ( 
    V := vertexSet G;
    I := binomialEdgeIdeal G;
    K := minors(2,transpose genericMatrix(ring I,#V,2));
    D := gens ideal(flatten mingens(quotient(I,K)/I));
    N := flatten entries D;
    F := apply(N,x-> if not quotient(I,x)==K then 0 else x)-set{0};
    min flatten degrees ideal(flatten F)
    )
\end{verbatim}
\end{procedure}


\section*{Acknowledgments}
 The authors thank Rafael H. Villarreal and Hongmiao Yu for useful discussions. We are also grateful to Rafael H. Villarreal for suggesting \emph{Macaulay2} code \ref{v-number-procedure} and for valuable comments on a previous version of this work. 
\nocite{*}
\bibliographystyle{alpha}
\bibliography{References}

\end{document}